\documentclass[review,onefignum,onetabnum]{siamart171218}

\usepackage{mathtools}

\usepackage{amsmath}
\usepackage{amssymb}
\usepackage{graphicx}

\usepackage{soul}

\usepackage{amsmath, appendix, ulem}
\usepackage{amssymb, color}

\usepackage{mathrsfs}
\usepackage{graphicx}
\usepackage{subfig}
\usepackage{float}
\usepackage{epsf}
\usepackage{titletoc}

\numberwithin{equation}{section}

      \newtheorem{assumption}{Assumption}
      
      \newtheorem{remark}{Remark}

\newcommand{\rd}{\,\mathrm{d}}
\newcommand{\ri}{\text{i}}

\newcommand{\rev}[1]{\textcolor{black}{#1}}

\title{Sensitivity analysis of Burgers' equation with shocks 
\thanks{
\funding{The work of Q.L.~is supported in part by National Science Foundation under the grant
  DMS-1619778, DMS-1740707, and DMS-1107291: RNMS KI-Net. The work of J.L. is supported in part by National
  Science Foundation under the grant DMS-1812573 and DMS-1107444: RNMS KI-Net. The work of R. S. is supported in part by National Science Foundation under the grant DMS-1107291: RNMS KI-Net.}}}

\author{Qin Li\thanks{Mathematics Department, University of Wisconsin-Madison, 480 Lincoln Dr., Madison, WI 53706 USA.
		\email{qinli@math.wisc.edu}}
\and Jian-Guo Liu\thanks{Department of Mathematics and Department of Physics, Duke University, Durham, NC 27708 USA. \email{jliu@phy.duke.edu}}
\and Ruiwen Shu\thanks{Department of Mathematics, University of Maryland, College Park, 4176 Campus Dr., College Park, MD 20742 USA. \email{rshu@cscamm.umd.edu}} }
		
\date{\today}

\begin{document}
\maketitle

\begin{abstract}
Generalized polynomial chaos (gPC) method has been extensively used in uncertainty quantification problems where equations contain random variables. {For gPC to achieve high accuracy, }PDE solutions need to have high regularity in the random space, but this is what hyperbolic type problems cannot provide. We provide a counter-argument in this paper, and show that even though the solution profile develops singularities in the random space, which {destroys the spectral accuracy} of gPC, the physical quantities (such as the shock emergence time, the shock location, and the shock strength) are all smooth functions of the uncertainties coming from both initial data and the wave speed: with proper shifting, the solution's polynomial interpolation approximates the real solution accurately, and the error decays as the order of the polynomial increases. {Therefore this work provides a new perspective to ``quantify uncertainties" and significantly improves the accuracy of the gPC method with a slight reformulation.} We use the Burgers' equation as an example for thorough analysis, and the analysis could be extended to general conservation laws with convex fluxes. 
\end{abstract}

\section{Introduction}
{Hyperbolic conservation laws describe \rev{many} important physics balance laws such as conservation of mass, momentum, and energy. \rev{They describe} various important continuum physics  including wave propagation and wave interactions.} It is a very classical mathematical subject that has a long tradition tracing back to Euler. In all these studies, the equations are deterministic, with prescribed boundary and initial conditions. Typically there are parameters in the equations that are simply pre-determined using constitutive laws. However, from a realistic point of view, {uncertainties are generic, in the sense that the initial/boundary conditions and equation parameters usually come from experiments and therefore inevitably have measurement error.} If the initial/boundary conditions or the constitutive laws are uncertain and inaccurate, is the solution affected dramatically by such uncertainties? And how does one quantify the influence of the uncertainties on the solution?

One particular example is from ocean science in which scientists need to determine the arrival time of a tsunami at a particular location (the land, for example). Such quantity is affected by the time and location of the explosion (an underwater earthquake, or underwater landslides or volcanoes), the undersea topography, the strength of wind and many others. In practice, we only have limited information about them, and mathematically it is natural to model the unknowns as uncertain parameters in the equations. It is then a mathematical question to understand how the solution behaves as the parameters change the value, and assess the associated sensitivities.

Suppose we use the 1D shallow water wave equation to model tsunami:
\begin{equation}\label{sweq}
\left\{\begin{split}
& \partial_t h + \partial_x(hu) = 0\,, \\
& \partial_t (hu) + \partial_x (hu^2 + \frac{1}{2}h^2) = 0\,,
\end{split}\right.
\end{equation}
where $x$ and $t$ are space and time coordinates, $h$ is the depth of water, and $u$ is the velocity of the seawater. The explosion that triggers the tsunami is typically modeled by a shock profile in the initial data $(h_\text{in},u_\text{in})$. The data is certainly unknown but we can assume the initial data depends on a random variable \rev{$Z:\Omega\rightarrow\mathbb{R}^d$ that lives in a probability space $(\Omega,\mathcal{B},\mathbb{P})$. The joint probability density function of the random vector $z=Z(\omega)\in\mathbb{R}^d$ is denoted as $\pi(z)$.} There are many physical quantities that are of interest. One example is the arrival time of the tsunami to $x_0$, the location of the land, denoted by $t^\sharp$. The ultimate goal is to predict
$$\mathbb{E}(t^\sharp) = \int t^\sharp(z)\pi(z)\rd{z}\,,$$
the expected arrival time (assuming the random variable $z$ in the initial data correctly describe the explosion), and
$$\textnormal{var}(t^\sharp) = \int (t^\sharp(z)-\mathbb{E}(t^\sharp))^2\pi(z)\rd{z}\,,$$
the variance that quantifies the reliability of the prediction.

It is a standard procedure to simplify the model (\ref{sweq}) using the two Riemann invariants $u\pm 2\sqrt{h}$~\cite{landaulifshitz}. The equations now read:
\begin{equation*}
\partial_t (u\pm 2\sqrt{h}) + (u\pm \sqrt{h}) \partial_x (u\pm 2\sqrt{h}) = 0\,.
\end{equation*}
Suppose one cares about one Riemann invariant $v = u+2\sqrt{h}$, and set $(u-2\sqrt{h})|_{t=0}=c(z)$ to have $z$ dependence, equation~\eqref{sweq} is then reduced to
\begin{equation*}
\partial_t v + (\frac{3v}{4} + \frac{c}{4}) \partial_x v = 0\,,
\end{equation*}
which can be reformulated into the form of the Burgers' equation
\begin{equation}\label{eq}
\partial_t u + \partial_x (\frac{\alpha(z)}{2}u^2) = 0\,,\quad u(t=0,x,z) = u_\text{in}(x,z)\,,
\end{equation}
whose wave speed $\alpha(z)u$ and the initial condition vary according to the initial condition of~\eqref{sweq}. Our goal then is to study $u$, or some physical quantities derived from $u$, such as $t^\sharp$'s dependence on $z$.

There are several conventional ways of computing {the solution to} equations with unknown parameters. Among them, the generalized polynomial chaos (gPC) method has been quite popular during the recent  {several years}. {To} a large extent, it can be regarded as a spectral-type method applied onto the $z$-space.  {Although having the same way of representing functions by orthogonal polynomial expansions, the gPC method} has many variations (such as gPC-stochastic Galerkin, gPC-stochastic collocation (gPC-SC), and gPC-sparse grid etc.)~\cite{GHANEM1996289,XiuHesthaven_collocation,Xiu_gpc,nobile_sparse_2008}. Take gPC-SC method for example: a few sample points $\{z_j\}_{j=1}^N$ are pre-selected according to the probability distribution of $\pi(z)\rd{z}$ (typically one uses collocation points), and with these $z_j$ fixed, the equations are deterministic and can be easily computed by existing numerical methods. Upon getting the solutions at these preset sample points, the solutions of the equation on other configurations of $z$ are then interpolated in a polynomial way. The interpolation is regarded as an accurate surrogate to the true solution.

The method gains its popularity largely due to its ``spectral" nature: in many cases it gives spectral convergence, which is faster than most other methods. But it also inherits the strong requirement on the data: for the spectral convergence to be valid, regularity of the solution has to be justified: only when the to-be-interpolated functions are shown to be smooth can one prove the fast convergence (depending on the regularity). For ``lucky" cases like elliptic or parabolic equations, one can show such regularity, as was done in~\cite{babuska_galerkin_2004,babuska_stochastic_2007,ZhangGunzburger12}, but it is not the case for most hyperbolic conservation law equations~\cite{majda}. On the contrary, it is a well-known fact that the Burgers' equation, the toy model equation in scalar conservation laws, develops, in finite time, singular points (or shocks as they are termed) even with $C^\infty$ initial data, and such singularity in $x$ will  {naturally result in the development of discontinuity in $z$}. More importantly, such irregularity is generic. Because of this, {although the convergence of gPC expansions may still be guaranteed~\cite{EMSU},} there is no hope to declare any results on the {spectral} accuracy {of gPC methods, which is based on that of polynomial expansions}~\cite{gunzburger_stochastic_2014,Xiu10}.

We present some preliminary computation in Figure~\ref{fig:reference_ex}, which shows our numerical results of the Burgers' equation with random initial data:
\begin{equation}\label{eqn:initial_example}
u_\text{in}(x,z) = v(x,z) - 0.2(v(x,z)+0.5)(1-v(x,z)^2)z,\quad \text{with}\quad v(x,z) = \frac{1-e^{x-3z^3}}{1+e^{x-3z^3}}\,,
\end{equation}
using gPC-SC method. In what follows, we sometimes omit the variable $z$ in functions if it is clear from the content. We have assumed the random variable $z\in [-1,1]$ satisfies the Chebyshev distribution. According to gPC-SC method, the Chebyshev quadrature points are selected:
\begin{equation}\label{eqn:quadrature}
z_j=\cos\left(\frac{2j-1}{2N_z}\pi\right),\quad j=1,\dots,N_z=10\,,
\end{equation}
and the equation is numerically solved at each $z_j$. We run the experiment up to $T=2.2$, and collect $u(t=2.2,x,z_j)$ for all $z_j$ (5th order WENO scheme with the 3rd order SSP Runge-Kutta in time is used to minimize discretization error from time and space). The solution to all other $z\in[-1,1]$ is then interpolated using a 9-th order polynomial. We clearly see the spurious oscillations in Figure~\ref{fig:reference_ex}, where we compared the interpolated solution of $u(t=2.2,x,z_0=0.234)$ with the true solution as a function of $x$, and the interpolated solution of $u(t=2.2,x=-0.5,z)$ with the true solution as a function of $z$.

\begin{figure}
	\includegraphics[width=0.43\textwidth,height=0.2\textheight]{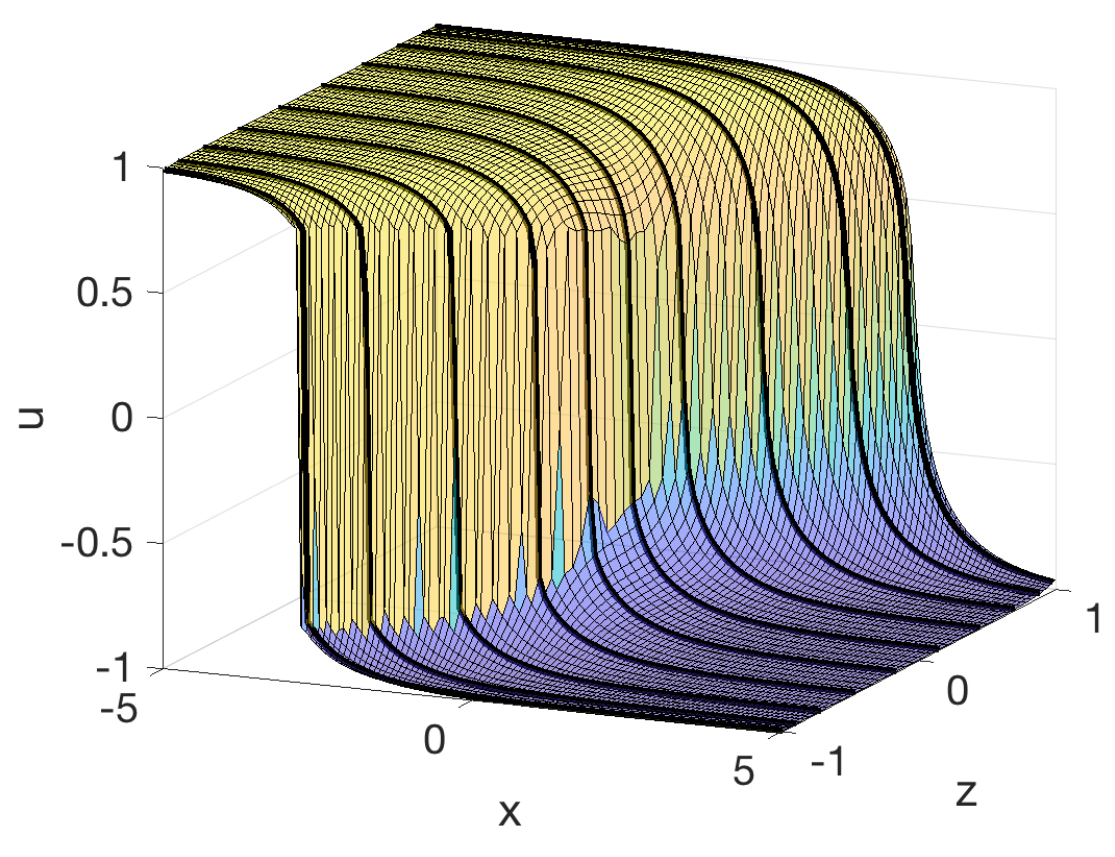}
	\includegraphics[width = 0.43\textwidth, height = 0.2\textheight]{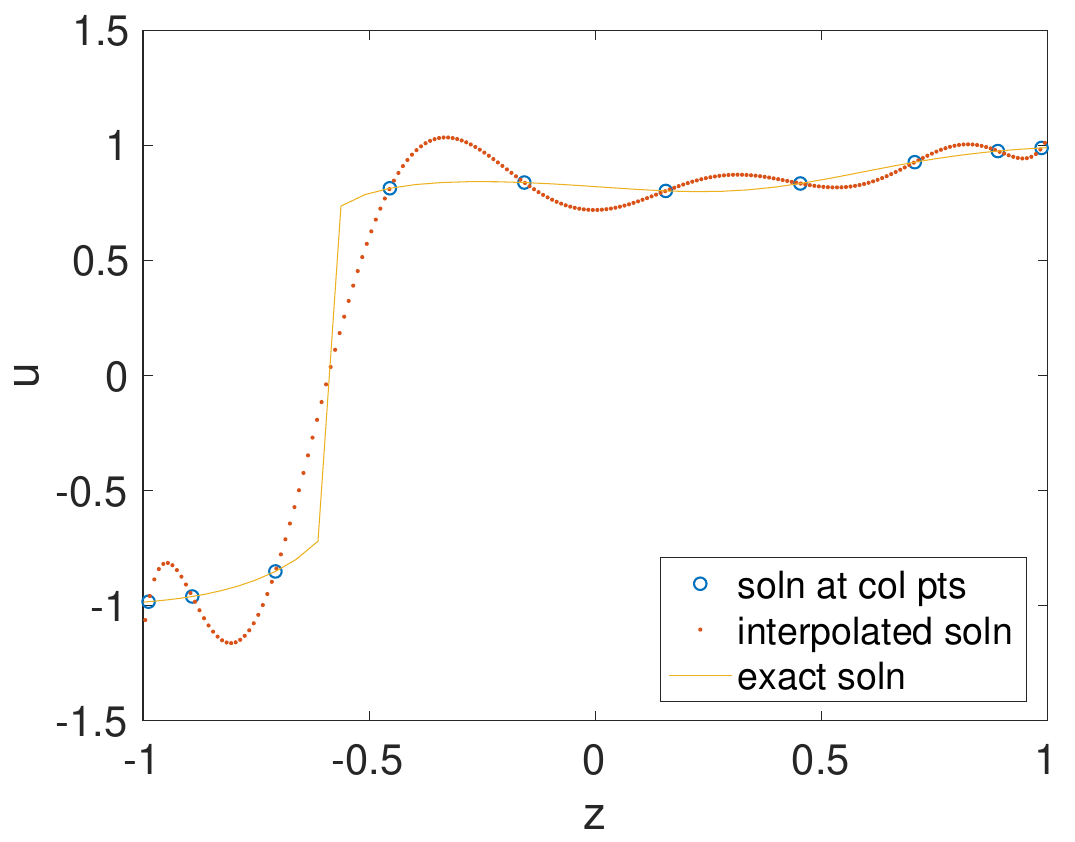}\\
	\includegraphics[width=0.43\textwidth,height=0.2\textheight]{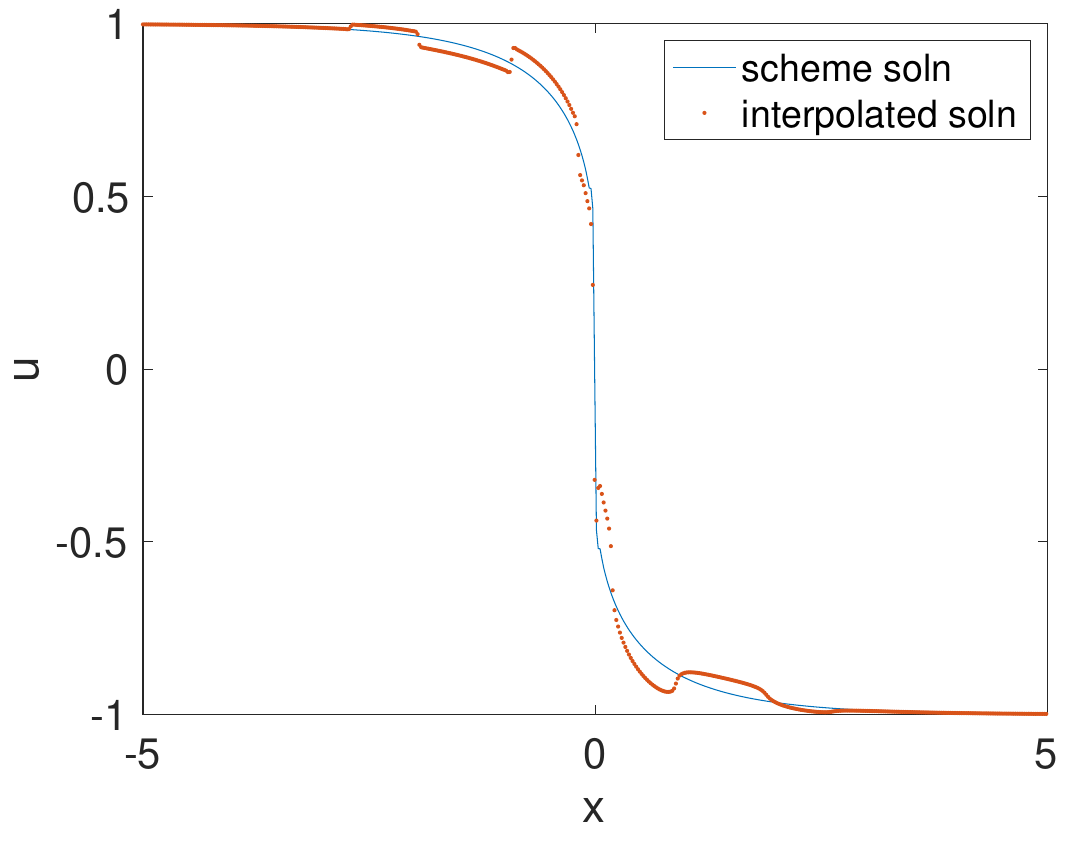}
	\includegraphics[width=0.43\textwidth,height=0.2\textheight]{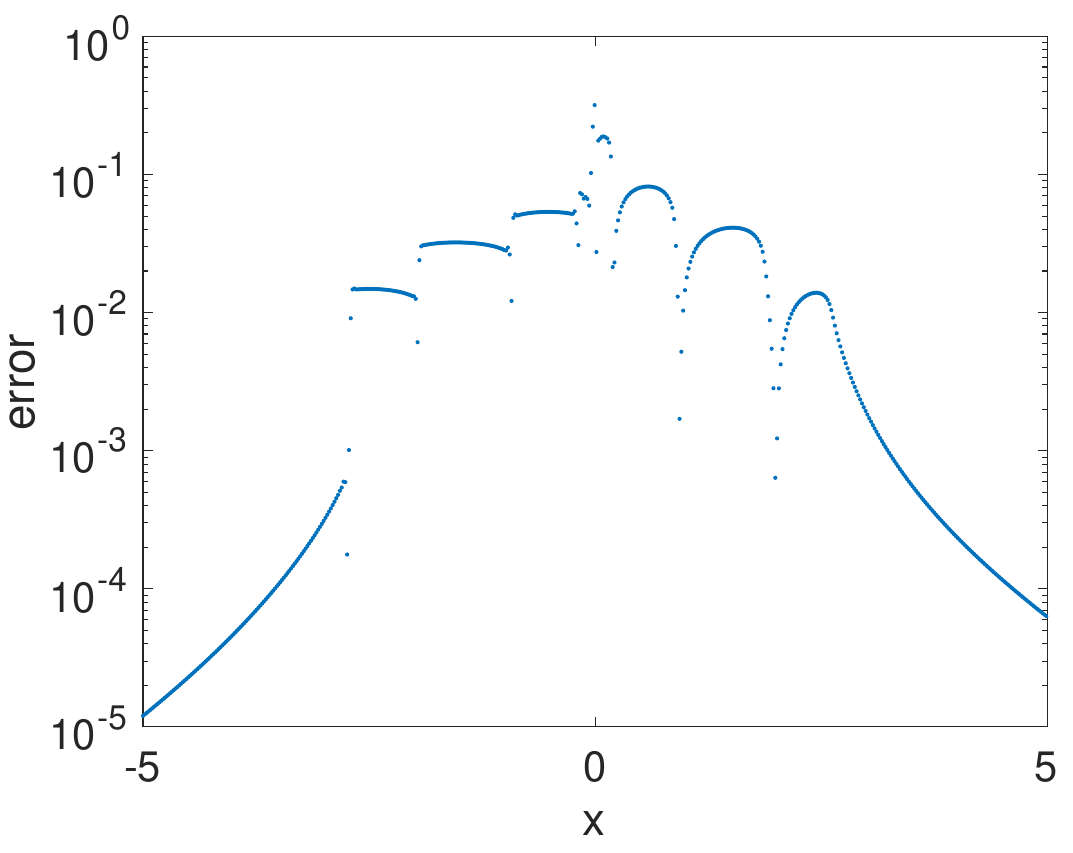}
	\caption{Top left: the numerical scheme solutions at sample points $\{z_j\}$; top right: the result by direct polynomial interpolation at $t=2.2,\,x=-0.5$; bottom left: compare the direct interpolation solution (dots) with the numerical scheme solution (line) at $z_0=0.234$; bottom right: error (difference between the two solutions in the middle picture).}\label{fig:reference_ex}
\end{figure}

\rev{Such loss of spectral accuracy is expected: due to the Gibbs phenomenon, the spectral method, albeit performing well in $L_2$, drastically fail in $L_\infty$, in the sense that it gives inaccurate interpolations to discontinuous functions like this one, and the $L_\infty$ error does not decay when one increases the order of the interpolation.} This poor numerical performance was noted in earlier works of~\cite{gunzburger_stochastic_2014,Xiu10,majda,Schwab_CL,Welper}, and largely for this reason, the results obtained using the gPC type methods are regarded unsatisfactory in the hyperbolic conservation laws setup.

\rev{However, we provide a counter-argument in the current paper. In particular, we will be dealing with the inaccuracy in the $L_\infty$ sense. It is based on a simple observation: even though the solution $u(t, x, z)$ may be discontinuous in $z$ for all $x$, making the interpolation severely inaccurate, the physical quantities that ``practically" matter are still smooth functions in $z$-space. They do not contain jump discontinuities in the random space, and we view them being insensitive to the random perturbation in the initial/boundary conditions. Such physical quantities include the shock location, the shock strength, the shock emerging time and the arrival time of the shock at certain locations. With these quantities identified, we can perform suitable ``shifting" of the solutions, which permits point-wise accuracy, meaning the gPC interpolation can produce accurate approximations even in the $L_\infty$ sense.
}
%

To be more precise, for the one-shock solutions, denote (as plotted in Figure~\ref{fig:ushape}):

\begin{enumerate}
\item $t^*(z)$, the shock appearing time;
\item $x^c(t,z)$, the shock location that moves with respect to time for $t\geq t^\ast$;
\item $u_1(t,z)-u_2(t,z)$, the shock strength (for $t\ge t^*(z)$), with 
\begin{equation}
u_1(t,z)=\lim_{x\rightarrow x^c(t,z)-}u(t,x,z),\quad\text{and}\quad u_2(t,z)= \lim_{x\rightarrow x^c(t,z)+}u(t,x,z)\,,
\end{equation}
being the upper and lower boundaries of the shock;
\item $t^\sharp(z)=\inf\{t: x^c(t,z)\ge x_0\}$, the shock arriving time\footnote{In case $x^c(t)\ge x_0$ is never satisfied, $t^\sharp$ is understood as infinity.}. Here $x_0$ denotes the predetermined location of land.
 \end{enumerate}

We repeat the previous example, and \rev{ find numerical evidence that shows these physical quantities are indeed smooth functions in $z$, seen in Figure~\ref{fig:u_12_x_c}.}

\begin{figure}
\centering
	\includegraphics[width=0.8\textwidth,height=0.25\textheight]{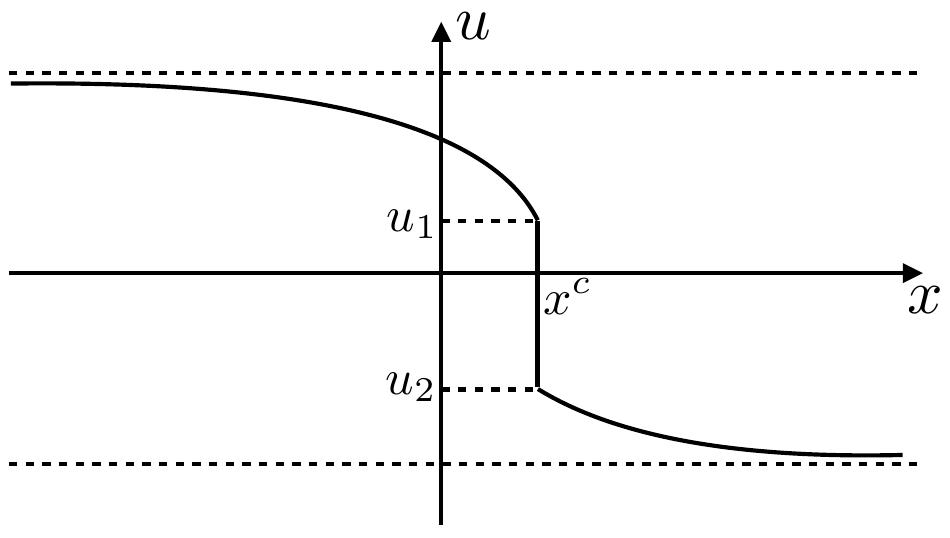}
	\caption{A demonstration of the quantities $u_1,\,u_2,\,x^c$, in the case of one-shock solution.}\label{fig:ushape}
\end{figure}

\begin{figure}
\centering
	\includegraphics[width=0.32\textwidth,height=0.2\textheight]{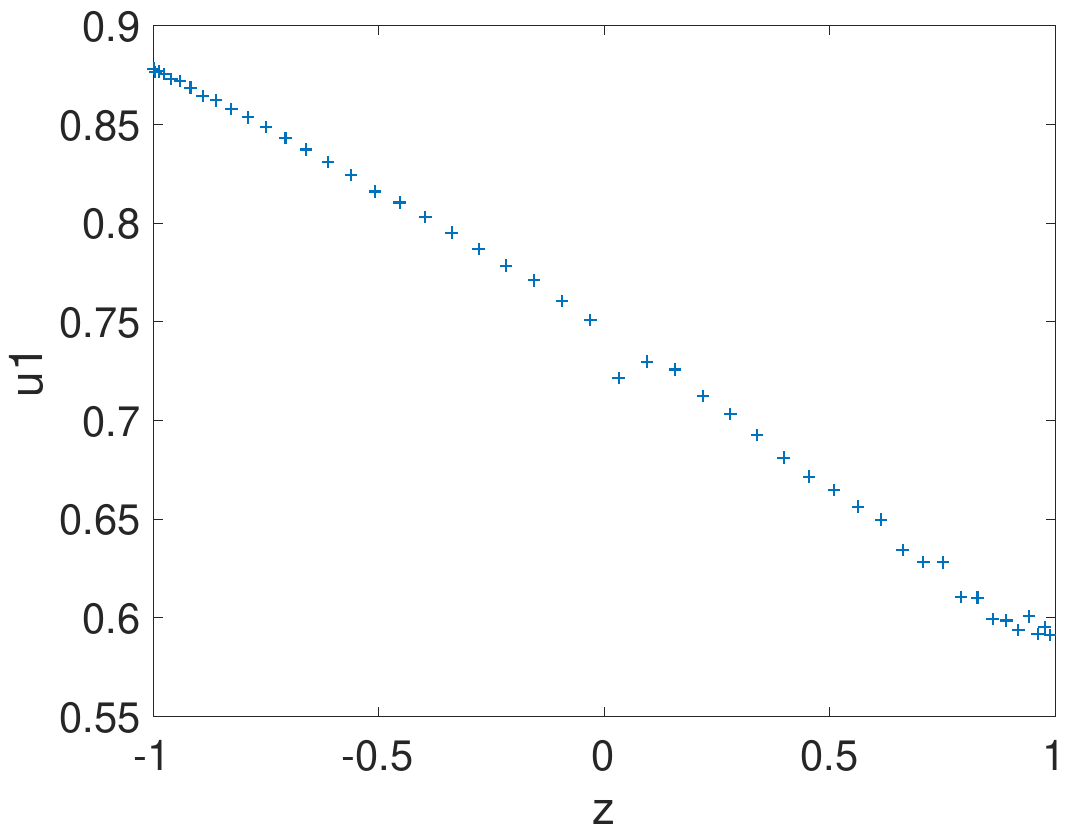}
	\includegraphics[width=0.32\textwidth,height=0.2\textheight]{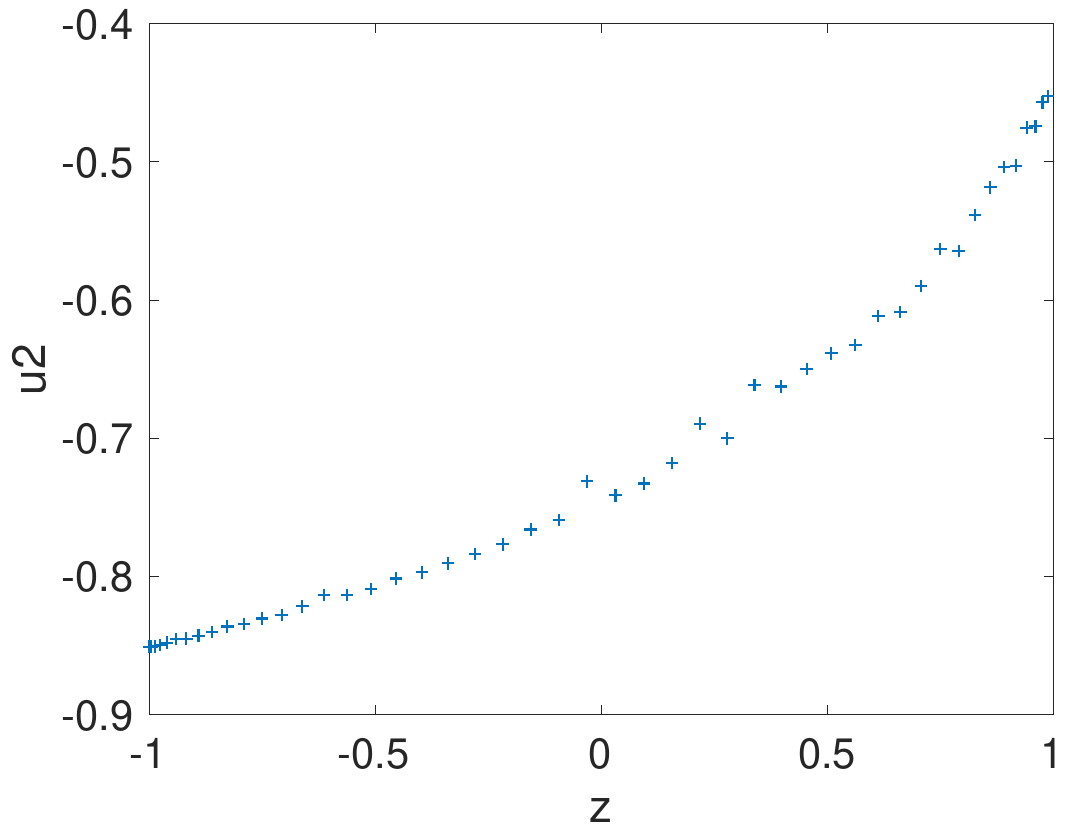}
	\includegraphics[width=0.32\textwidth,height=0.2\textheight]{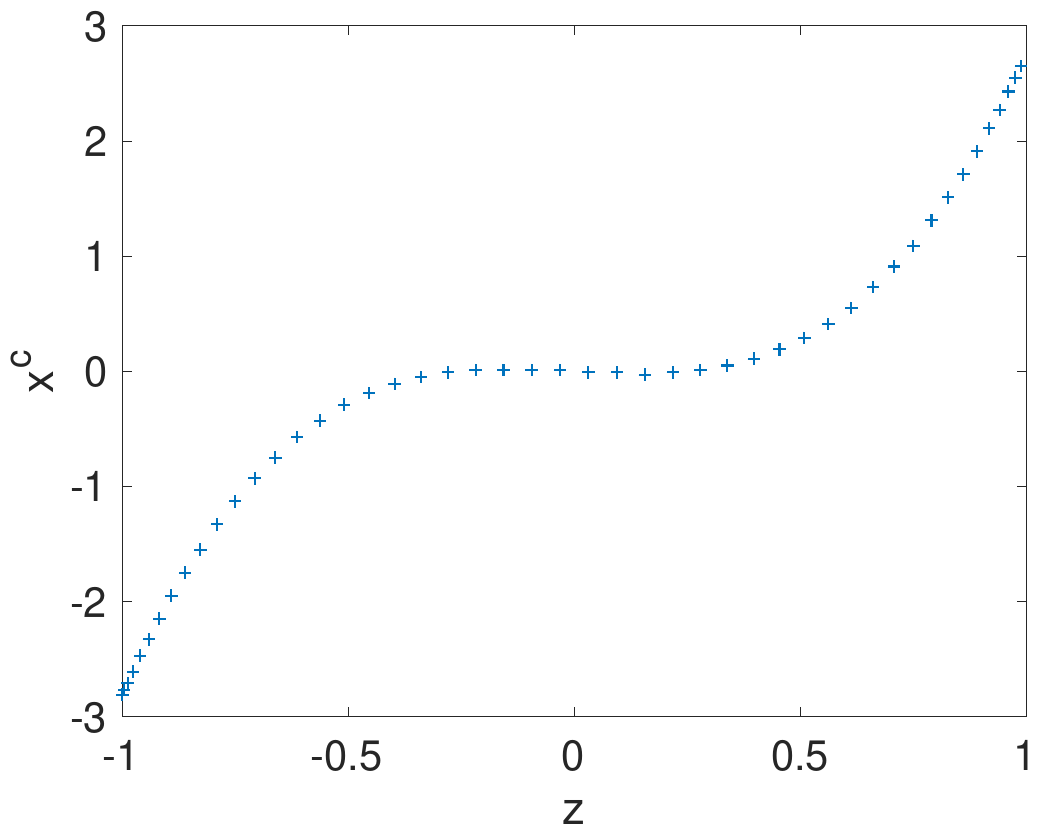}
	\caption{Left to right: $u_1,\,u_2,\,x^c$ as functions of $z$. The small zigzags in plots are from numerical error. {Numerically to identify these quantities, we look for the biggest jump between two adjacent grid points. This procedure brings some error and is not easy to be eliminated by adopting other numerical solver for the conservation law.}}\label{fig:u_12_x_c}
\end{figure}

The aim of the current paper is to mathematically prove this observation. Our main claim in this paper is: even the solution to the scalar hyperbolic conservation law varies drastically with respect to the random inputs in initial condition or the equation parameters, the physical quantities are insensitive to them, or more precisely:
\begin{theorem}[formal statement]\label{thm_main}
Let $u(t,x,z)$ be the solution to the Burgers' equation (\ref{eq}) a parameter $z$ representing uncertainty. Assume the initial data $u_\text{in}(x,z)$ is smooth and satisfies a set of conditions (to be made precise later) so that only one shock will appear for each $z$, and assume $\alpha$ is smooth on $z$. Then
\begin{enumerate}
\item The shock appearing time $t^*(z)$ depends smoothly on $z$;
\item The shock location $x^c(t,z)$ depends smoothly on $z$;
\item The shock strength $u_1(t,z)-u_2(t,z)$  depends smoothly on $z$;
\item The shock arriving time $t^\sharp(z)$, depends smoothly on $z$.
\end{enumerate}
\end{theorem}

As a direct corollary of this theorem, we also find that if the solutions are ``shifted correctly", the shifted solution $\tilde{u}(t,x,z)$ becomes smooth function in $z$ for every $t$ and $x$, granting the accuracy to the gPC type method. This could be summarized as:
\begin{theorem}[formal statement]\label{thm_tilde}
Define the shifted solution:
$$
\tilde{u}(t,x,z) = u(t+t^*(z),x+x^c(t+t^*(z),z),z)\,,
$$
so that shocks are aligned for all $z$ to the same emerging time and location, then with the same assumptions as Theorem \ref{thm_main}, $\tilde{u}$ is smooth in $z$ away from the set $\{(t,x,z):\,\,x=0\}$ for $t> 0$.
\end{theorem}

There are many groups of researchers working on similar topics. In~\cite{hou_li_zhang16_2,hou_li_zhang16_1}, the authors adopted the patch-wise low-rank studies. In~\cite{Siddhartha_random_flux,Mishra_Schwab}, the authors proved the wellposedness of entropy solution when randomness is present in initial data and flux, and $L_1$ contraction is used for estimating the error from interpolation method.  In~\cite{Schwab_CL}, the authors gave a very detailed analysis on the shock location of the Burgers' equation with Heaviside function as the initial data. A similar approach was taken in~\cite{Welper} where the author presents very powerful numerical evidence that demonstrates the shifting indeed ``save" the regularities of the solutions. Another approach proposed in~\cite{Mishra2017} and the reference therein is to explore Monte Carlo (and Multi-level Monte-Carlo) methods. In~\cite{Despres_UQ_CL}, the authors introduce entropic variables and expand the solution as polynomials of the new variable with the understanding that it is smoother when represented by the new variable. \rev{In~\cite{CVK} the authors derive the reduced order equations for the one-point and two-point probability density functions of the solution field, and design algorithms to compute the statistical properties of the random shock wave accordingly.} Other approaches include~\cite{Abgrall_truncate_encode,Abgrall2016} where authors either employed the so-termed truncate-encode framework, or in~\cite{DP_uncertaity} where kinetic formulation is utilized. The approach we are taking is in line with~\cite{Welper} and~\cite{Schwab_CL} but we emphasize on giving a quantitative mathematical justification in general cases. {Our approach is also closely related to the stochastic transformation proposed in~\cite{HLRZ} for the viscous Burgers' equation with random forcing.}

{We remark that Theorem \ref{thm_main} and \ref{thm_tilde} also hold for general scalar conservation laws with smooth convex flux functions. This means that the smoothness in $z$ of the physical quantities (shock appearing time, shock location, shock strength, etc.) is a generic fact, and the smoothness in $z$ of the solution profile can always be recovered by shifting correctly. }

We emphasize before finishing the introduction that the goal of the paper is not to justify the use of gPC method on Burgers' equation, or hyperbolic conservation laws in general, but to bring one more aspect to understand the shock structure in wave-like equations when uncertainties \rev{are} present. In fact, almost all numerical methods somewhat rely on the regularity of some to-be-computed quantities, and the result obtained in this paper serves as a justification for these algorithms applied in $z$ space.

The rest of this paper is organized as follows: in Section 2 we introduce some notations and state the precise quantitative version of Theorem \ref{thm_main} and \ref{thm_tilde}; in Section 3 we focus on the deterministic case and prepare some necessary tools for analyzing $u_1$ and $u_2$; these tools are crucial in Section 4 where we prove Theorem \ref{thm_main} and~\ref{thm_tilde}. {We also extend the results to treat conservation laws with general convex fluxes in Section 5.} \rev{Some proofs are tedious and not essential to the main contexts, and they are left in Appendix.}

\section{Notations and precise statement of main results}
There is a \rev{large variety} of solution behavior of conservation laws, and we restrict ourselves to the class of smooth initial data such that only one shock is developed for $t>0$. Mathematically:
\begin{assumption}\label{ass:x_u_initial}
denote $u_\text{in}(x,z)$ the initial data. \rev{We require $u_\text{in}(x,z)$ to be smooth in $(x,z)$, and $u_{in}(\cdot,z)$, as a function of $x$, to satisfy the following:}
\begin{itemize}
\item $u_\text{in}$ monotonically decreases in $x$: i.e. $u'_\text{in}(x)<0$ for all $x$, and $\lim_{x\rightarrow \pm\infty}u_\text{in}(x) = u_\pm$. Here $u_+ < u_-$ are constants independent of $z$;
\item $u_\text{in}$ has a unique inflection point $(x^\ast,u^\ast)$, meaning $u_\text{in}(x^\ast) = u^\ast$ and $u_\text{in}''(x^\ast) = 0$;
\item $u'''_\text{in}(x^\ast)>0$;
\item $\alpha(z)$ has uniform bounds: $0 < \alpha_0 <\alpha(z) < \alpha_1$.
\end{itemize}
\end{assumption}

Under these assumptions, we restate Theorem \ref{thm_main} rigorously and quantitatively:
\begin{theorem}\label{thm_main1}
Let $u(t,x,z)$ be the solution to the Burgers' equation (\ref{eq}) with uncertainty. Assume that the initial data $u_\text{in}(x,z)$  satisfies Assumption \ref{ass:x_u_initial}, and that $u_+ +\delta \le u^*(z) \le u_- -\delta$ for all $z$, with $\delta>0$. Then with $\alpha(z)>0$ being smooth in $z$, one has:
\begin{enumerate}
\item The shock appearing time is given by 
\begin{equation}\label{tst}
t^*(z)=-\frac{1}{\alpha (z)u_\text{in}'(x^*(z),z)}\,,
\end{equation}
 where $x^*(z)$ is as in Assumption \ref{ass:x_u_initial}. It follows that $t^*(z)$ depends smoothly on $z$.
\item The shock location $x^c(t,z)$ (defined for $t\ge t^*(z)$) depends smoothly on $z$, and satisfies the estimate
\begin{equation}
\partial_z^k x^c = \mathcal{O}((t-t^\ast)^{\min\{3/2-k,0\}})\,,
\end{equation}
for $t-t^\ast$ small enough.
\item The shock strength $u_1(t,z)-u_2(t,z)$ depends smoothly on $z$, and satisfies the estimate
\begin{equation}
\partial_z^k (u_1-u_2) =  \mathcal{O}((t-t^\ast)^{1/2-k}),
\end{equation}
for $t-t^\ast$ small enough.
\item Let $x_0$ be large enough so that $x_0>\sup_z x^c(t^*(z),z)$. Assume
\begin{equation}\label{214assu}
\partial_t x^c(t^\sharp,z)\ne 0\,,
\end{equation}
and that $t^\sharp<\infty$, then $t^\sharp$ depends smoothly on $z$.
\end{enumerate}

\end{theorem}
\begin{remark}
We now comment that Assumption~\eqref{214assu} is not restrictive. In fact, as will be seen in Section \ref{sec_det}, $\partial_t x^c$ has the explicit expression (\ref{eqn:center_location}), and thus~\eqref{214assu} can be checked explicitly. In the same section, we can also derive that in long time,
\begin{equation}
\lim_{t\rightarrow \infty} \partial_t x^c(t,z) = \alpha\frac{u_-+u_+}{2}\,,
\end{equation}
and thus $u_-+u_+>0$ automatically leads to~\eqref{214assu} for large $t$. This exactly corresponds to the realistic case when a tsunami forms far away from the land, and takes a long time to propagate to the land.
\end{remark}

Theorem \ref{thm_tilde} is a corollary of the theorem above, and in the precise form it states:
\begin{theorem}\label{thm_tilde1}
Under the same assumptions as Theorem \ref{thm_main1}, the translated solution 
\begin{equation}\label{tildeu}
\tilde{u}(t,x,z)=u(t+t^*(z),x+x^c(t+t^*(z),z),z)\,,
\end{equation}
 is smooth in $z$ away from the set $\{(t,x,z):\,x=0,\,t=0\}$, and has the estimate
\begin{equation}
|\partial_z^k \tilde{u}(t,x,z)| =\mathcal{O}(|x|^{1-2k}t^{\min\{3/2-k,0\}}),
\end{equation}
if $t> 0$ is small enough.
\end{theorem}

{This theorem implies that a proper shifting of the solution could eliminate the irregular jumps in the solution, and this would allow the spectral type method such as gPC to apply well. In particular, using the same example as in Section 1, assuming the random variable $z\in [-1,1]$ satisfying the Chebyshev distribution,  we denote:}
\begin{equation}\label{poly_interp}
u^N(t,x,z) = \sum_{j=0}^N\tilde{u}(t,x,z_j)\ell_j(z)\,,
\end{equation}
where $z_j$ are the Chebyshev quadrature points defined in~\eqref{eqn:quadrature}, and $\ell_j$ are the corresponding Lagrange polynomials in $z$ domain. {Then we have the following theorem. We note that if $z$ satisfies another distribution, similar techniques can still be applied with $z_j$ shifted accordingly. For the conciseness of the statement here we stick to this particular kind of random variable.}

\begin{theorem}\label{thm_interp0}
{Assume the random variable $z\in [-1,1]$ satisfying the Chebyshev distribution.} Under the same assumptions as Theorem \ref{thm_main1}, the error of the interpolated solution $u^N$ can be estimated by
\begin{equation}\label{thm_interp_10}
|\tilde{u}(t,x,z)-u^N(t,x,z)| \le \frac{C(m)|x|^{-1-2m}t^{1/2-m}}{N^m},\quad \forall z\in [-1,1],\quad {t>0}\,,
\end{equation}
for any $m\ge 1$, i.e., it has $m$-th order accuracy away from the shock location and the shock appearing time.

Furthermore, if we use $\mathbb{E}(u^N),\,\textnormal{var}(u^N)$ to approximate the mean and variance of $\tilde{u}$, and assuming that $z\in[-1,1]$, then we have the error estimate
\begin{equation}\label{thm_interp_20}
|\mathbb{E}(\tilde{u})-\mathbb{E}(u^N)| \le \frac{C(m)|x|^{-1-2m}t^{1/2-m}}{N^m}\,,
\end{equation}
\begin{equation}\label{thm_interp_30}
 |\textnormal{var}(\tilde{u})-\textnormal{var}(u^N)| \le \frac{C(m)(1+\min\{\|\tilde{u}-u^N\|_{L^\infty_z},N^2\})|x|^{-1-2m}t^{1/2-m}}{N^m}\,.
\end{equation}
\end{theorem}


{We remark that all the three estimates in Theorem \ref{thm_interp0} are pointwise in $t$ and $x$. They deteriorate as $|x|$ or $t$ gets small, i.e., the location is close to the shock or the time is close to the shock appearing time.}

Note that the $\min$ in the estimate of $\textnormal{var}(\tilde{u})-\textnormal{var}(u^N)$ in~\eqref{thm_interp_30} is necessary. For fixed $x$, as $N$ goes to infinity, according to~\eqref{thm_interp_10}, $\tilde{u}(t,x,z)-u^N(t,x,z)$ shrinks to zero, but on the other hand, for fixed $N$ and $x\sim 0$ (close to the shock), the difference between $\tilde{u}$ and $u^N$ could be significant and we use $N^2$ as the bound there.

We note that Theorem~\ref{thm_main1} and Theorem~\ref{thm_tilde1} only state the smoothness in $z$-space regarding $z$ as an external unknown parameter. It was not until Theorem~\ref{thm_interp0} where we incorporate the statistical behavior, and thus $\pi(z)$, the distribution is needed.

Proofs for Theorem~\ref{thm_main1} and Theorem~\ref{thm_tilde1} heavily depend on the delicate analysis of $u_1$ and $u_2$, while Theorem~\ref{thm_interp0} immediately follows the regularity results in Theorem~\ref{thm_tilde1}. Since the dynamics of $u_{1,2}$ are so crucial, for a clear presentation, we devote Section~\ref{sec_det} to developing the necessary tools for the equation in the deterministic setting. These results will be used in later sections, where energy estimate is used for showing the two main theorems.

\section{Burgers' equation -- deterministic case}\label{sec_det}
In this section we will mainly focus on the shock behavior and the main tool is the hodograph transform. The reformulation is performed in Section~\ref{sec:reformulation} and the local-in-time shock behavior is presented in Theorem~\ref{thm1_deterministic} in Section~\ref{sec:det_shock}.

\subsection{Reformulation of the Burgers' equation}\label{sec:reformulation}
The monotonicity assumption of $u$ on $x$ makes the application of the hodograph transform possible: by flipping $x$, $u$ coordinates we can study the evolution of $x(u)$ in time. Denote $x=x(t,u)$ the inverse function of $u(t,x)$ for all $t$, then the domain is $(t,u)\in [0,\infty)\times(u_+,u_-)$.

\subsubsection{Before the formation of the shock at time $t^\ast$}
$x(u)$ is the coordinate that sits on the $u$-level set. Since the Burgers' equation has its wave propagating with speed $\alpha u$, then before $t^\ast$, we have:
\begin{equation}\label{eq1}
\begin{cases}
\partial_t x(t,u)  = \alpha u\,,\quad u\in (u_+,u_-)\,,\\
 x(t=0,u) = x_\text{in}(u)\,.
\end{cases}
\end{equation}

Assumption~\ref{ass:x_u_initial} on $u_\text{in}$ could be translated to assumption on $x_\text{in}$:
\begin{itemize}
\item $x'_\text{in}(u)< 0$;
\item $x_\text{in}$ has one unique inflection point at $(u^\ast,x^\ast)$;
\item $(x_\text{in})'''(u^\ast)<0$.
\end{itemize}

Take $\partial_u$ of~\eqref{eq1}, we also have:
\begin{equation}\label{eq2}
\partial_t \partial_ux(t,u)  = \alpha \,,\quad \Rightarrow\quad\partial_ux(t,u) = x_\text{in}'(u) + \alpha t = -f(u)+\alpha t\,,
\end{equation}
where we used the notation
\[
f(u) = -x_\text{in}'(u)\,.
\]
Combined with the property $x'_\text{in}(u)< 0$, we have $\partial_ux(t,u)\leq 0$ for $t<t^\ast$ where
\begin{equation}
t^\ast = \min_u (-\frac{1}{\alpha }x'_\text{in}(u)) = -\frac{1}{\alpha}x'_\text{in}(u^\ast)\,,
\end{equation}
is the earliest time for a shock to emerge. Such a shock appears at $(u^\ast,x^\ast)$.

\subsubsection{After the formation of the shock at $t^\ast=-\frac{1}{\alpha }x'_\text{in}(u^\ast)$}
The strong solution to~\eqref{eq} breaks down, and equation~\eqref{eq1} no longer correctly characterizes the solution behavior. As the weak formulation and the entropy condition are used to replace the strong form to characterize $u(t,x)$ on the $x-u$ plane, a different set of equation is needed for $x(t,u)$ on the $x-u$ plane: to do that we first utilize the Rankine-Hugoniot condition. Denote $u_1(t)$ and $u_2(t)$ as the top and the bottom of the shock, then the shock speed is:
\begin{equation}\label{eqn:shock_speed}
s = \alpha \frac{u_1(t)^2/2 - u_2(t)^2/2}{u_1(t)-u_2(t)} = \alpha \frac{u_1(t) + u_2(t)}{2}\,,
\end{equation}
meaning the shock location $x^c(t)$ satisfies the ODE:
\begin{equation}\label{eqn:center_location}
\frac{\rd}{\rd t}x^c = \alpha \frac{u_1(t) + u_2(t)}{2}\,,\quad\text{with}\quad x^c(t^\ast) = x^\ast\,.
\end{equation}
Notice that under the hodograph transform, the shock in $u(t,x)$ becomes a flat region in $x(t,u)$, ranging from $u_2$ to $u_1$ and has height $x^c$. 

\begin{figure}
\centering
	\includegraphics[width=0.7\textwidth,height=0.4\textheight]{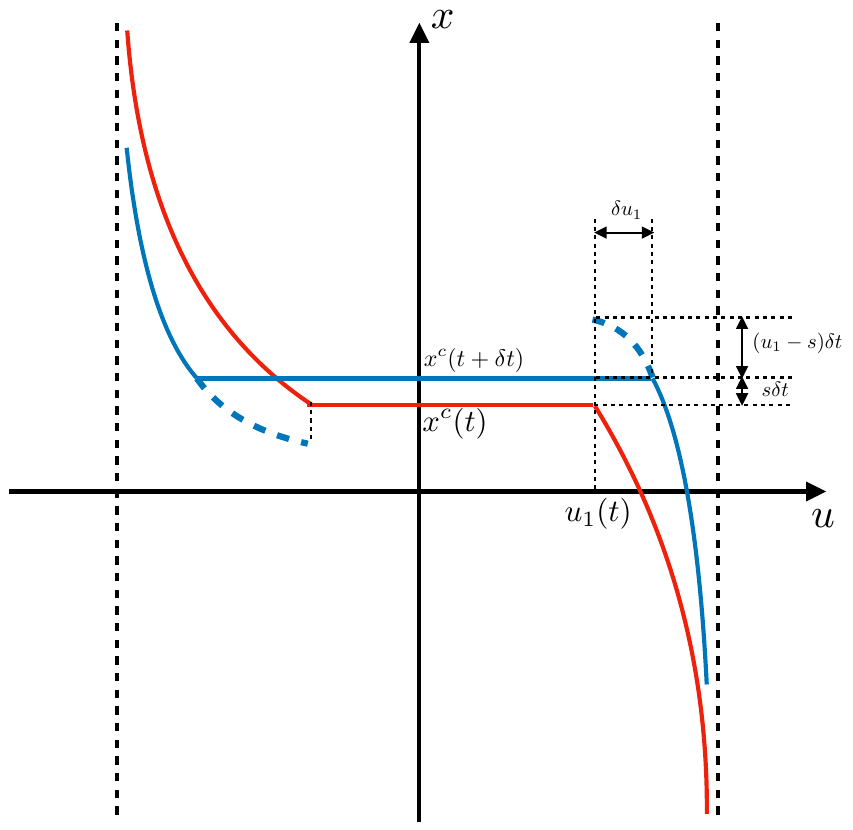}
	\caption{Evolution of $u_1,\,u_2$, in the hodograph-transformed picture. Here the red curve is the solution $x(t,u)$ at some time spot, and the solid blue curve is the solution after a small time period $\delta t$. The dashed blue curve is the dynamics of (\ref{eq1}).}\label{fig:u1u2pic}
\end{figure}

The ODE system for $u_{1,2}(t)$ can also be derived, as seen in Figure \ref{fig:u1u2pic}. If one focuses on the neighborhood of $u_1$, the flat region propagates in the vertical direction with speed $s$, while $x$ for $u>u_1$ propagates in the horizontal direction with a faster speed $u_1>s$. These coordinates that are supposed to travel faster then get absorbed into the flat region, widening it (around $u_1$) by:
 \begin{equation}
\delta u_1 = -\partial_xu(t,x)|_{u=u_1}\delta x=-\partial_xu(t,x)|_{u=u_1}\alpha \frac{u_1(t) - u_2(t)}{2}\delta t= -\frac{\alpha}{\partial_ux(t,u_1)} \frac{u_1(t) - u_2(t)}{2}\delta t\,,
\end{equation}
where
\begin{equation*}
\delta x = (\alpha u_1-s)\delta t = \alpha \frac{u_1(t) - u_2(t)}{2}\delta t\,,
\end{equation*}
presents the ``overshoot" before entropy condition is applied to ``cut" the multi-value solution. Considering~\eqref{eq2} and conduct the same analysis for $u_2$, one has
\begin{equation}\label{ode01}
\frac{\rd}{\rd t}\left(\begin{array}{c}u_1\\u_2\end{array}\right) = \left(\begin{array}{c}F_1(u_1,u_2)\\F_2(u_1,u_2)\end{array}\right) = \frac{\alpha}{2}(u_1-u_2)\left(\begin{array}{c}\frac{1}{f(u_1)-\alpha t}\\\frac{-1}{f(u_2)-\alpha t}\end{array}\right)\,,
\end{equation}
with the initial condition
\begin{equation}\label{odei0}
u_1(t^\ast)=u_2(t^\ast) = u^\ast\,.
\end{equation}
In the equation $F_{1,2}$ denotes the forcing terms for $u_{1,2}$ respectively.
\begin{remark}
Some comments are in line:
\begin{itemize}
\item According to~\eqref{eq2} and the monotonicity of $\partial_ux$, $-f(u_{1,2})+\alpha t=x'_\text{in}(u_{1,2})+\alpha t\leq 0$. Combined with~\eqref{ode01}, it is shown that $u_1(t)$ monotonically increases in time and $u_2(t)$ monotonically decreases in time, meaning:
\begin{equation}\label{cond}
u_1\geq u^\ast,\quad u_2\leq u^\ast,\quad f(u_{1,2})-\alpha t \geq 0\,.
\end{equation}
\item The system~\eqref{ode01} is self-consistent. This means the information in the shock is fully represented by $u_1$ and $u_2$. The general profile of $x(u)$ is irrelevant.
\end{itemize}
\end{remark}

\subsubsection{Summary}
To summarize the reformulation, the Burgers' equation, when written on the $(x,u)$-plane, becomes:
\begin{equation}
\begin{cases}
t<t^\ast = -\frac{1}{\alpha }x'_\text{in}(u^\ast):\quad&\text{Equation~\eqref{eq1}}\,,\\
t>t^\ast:\quad&
\begin{cases}
\text{Equation~\eqref{eq1}}\quad \text{with}\quad u\in(u_+,u_2)\cup(u_1,u_-)\,,\\
\text{Equation~\eqref{eqn:center_location}}\quad \text{with}\quad u\in (u_2,u_1)\,,
\end{cases}
\end{cases}
\end{equation}
where $u_1$ and $u_2$ are the shock locations satisfying the ODE system~\eqref{ode01}.

\subsection{Shock behavior for small time}\label{sec:det_shock}
The shock behavior is fully described by~\eqref{ode01}, which we study in depth in this section. To start, we first shift the coordinates and the time frame so that\footnote{This assumption implies that $x_\text{in}(u)\sim u^3$ for $u$ close to zero. In other words, $u_\text{in}(x)\sim (x-x^*)^{1/3}$ for $x$ close to $x^*$. This is known as the shock formulation profile for the Burgers' equation.} $u^\ast = 0$ and $t^\ast=0$. Since $t^\ast = -\frac{1}{\alpha }x'_\text{in}(u^\ast)$, $u_1(t^\ast) = u_2(t^\ast) = u^\ast$, one has:
\begin{equation}\label{odei}
u_1(0) = u_2(0) = 0\,,\quad\text{and}\quad f(0) = 0\,.
\end{equation}
Physically it means the flat region starts forming at $t=0$, $u=0$.

Assumption \ref{ass:x_u_initial} on $u_\text{in}$ is now formulated as the following:
\begin{assumption}\label{ass:x_u_initial2}
Denote $f(u)=-x_\text{in}'(u)$, then:
\begin{itemize}
\item $f(u)\geq 0$. This follows from item 2 of Assumption \ref{ass:x_u_initial} and the fact that $x_{\ri}'(u) = \frac{1}{u_\text{in}'(x)}$.
\item $f'(u) = 0$ at only one point $u^\ast$. To see this, one first differentiates $x_{\ri}'(u) = \frac{1}{u_\text{in}'(x)}$ to obtain $x_{\ri}''(u) = -\frac{u_\text{in}''(x)}{(u_\text{in}'(x))^3}$, and then notice item 3  of Assumption \ref{ass:x_u_initial}.
\item $f'(u) <0$ when $u<0$, $f'(u)>0$ when $u>0$, and $f''(0)>0$. The sign of $f'$ can be seen by $x_{\ri}''(u) = -\frac{u_\text{in}''(x)}{(u_\text{in}'(x))^3}$ and the signs of $u_\text{in}'(x),u_\text{in}''(x)$. The sign of $f''(0)$ can be seen by differentiating $x_{\ri}''$ and evaluating at $u=x=0$ to see $x_\text{in}'''(0) = -\frac{u_\text{in}'''(0)}{(u_\text{in}'(0))^4}$, and combining with item 4 of Assumption \ref{ass:x_u_initial}.
\end{itemize}
\end{assumption}
\begin{remark}
These assumptions, when combined with~\eqref{odei}, indicate that around $u=0$, $f(u)$ behaves like a quadratic function
\begin{equation}\label{au2}
f(u) \sim au^2\,,\quad \text{with} \quad a = \frac{1}{2}f''(0)>0\,.
\end{equation}
With this simpler quadratic form, in small time, $u$ is also small. The ODE system~\eqref{ode01} then gets simplified:
\begin{equation}\label{eqn:ode_approx}
\begin{cases}
\frac{\rd{u_1}}{\rd{t}} & = \frac{\alpha }{2}(u_1-u_2)\frac{1}{au^2_1-\alpha t}\,, \\
\frac{\rd{u_2}}{\rd{t}} & = -\frac{\alpha }{2}(u_1-u_2)\frac{1}{au^2_2-\alpha t} \,,
\end{cases}
\end{equation}
and one has the explicit solution:
\begin{equation}\label{ode1s}
u_1=-u_2 = (\frac{3\alpha}{a} t)^{1/2}\,.
\end{equation}
This result implies that $u_1$ and $u_2$ approximately grow in time with the power $\frac{1}{2}$. Generally, $f(u)$ is not a quadratic function, but one can still use two quadratic functions with different $a$ to sandwich the solution for a $t^{1/2}$ growth rate.
\end{remark}

We now state our theorem:

\begin{theorem}\label{thm1_deterministic}
	Under Assumption \ref{ass:x_u_initial2} on $f$, assume $u_{1,2}(t)$ solve the ODE system~\eqref{ode01} with initial condition~\eqref{odei} and  satisfy $u_1>0,\,u_2<0,\,f(u_{1,2})-\alpha t>0$ for $t>0$. Then, for any $\epsilon>0$, there holds
	\begin{equation}
	(\frac{3\alpha}{a+\epsilon} t)^{1/2}\le u_1 \le (\frac{3\alpha}{a-\epsilon} t)^{1/2},\quad -(\frac{3\alpha}{a-\epsilon} t)^{1/2}\le u_2 \le -(\frac{3\alpha}{a+\epsilon} t)^{1/2}\,,
	\end{equation}
	for $t$ small enough, with $a$ defined in~\eqref{au2}.
\end{theorem}

Note that the wellposedness of the system is not discussed in the theorem. In fact, away from the initial time, the forcing terms are Lipschitz, making the proof of the wellposedness standard, which we leave to Appendix~\ref{sec:appendix_ode}. To prove the small time behavior of $u_{1,2}$, we first start with an ODE (analyzed in Lemma~\ref{lem2.2}, and then utilize the symmetry condition (Lemma~\ref{lem:sym}) for a solution to the ODE system~\eqref{ode01} when $f$ is a quadratic function. The monotonicity (Lemma~\ref{lemcomp}) is then applied to sandwich the solution to the problem in which $f$ is not quadratic. Without loss of generality, $\alpha$ is set to be $1$ below.
\begin{lemma}\label{lem2.2}
Under Assumption \ref{ass:x_u_initial2}, and assume $u(t)>0$ and $f(u)>t$ for all $t>0$, the ODE
\begin{equation}\label{odeu}
\frac{\rd{u}}{\rd{t}} = \frac{u}{f(u)-t}\,,\quad u(0)=0\,,
\end{equation}
has a unique solution {given by the implicit function
\begin{equation}\label{odeu_soln}
t = \frac{1}{u}\int_0^u f(s)\rd{s}\,.
\end{equation}
}
\end{lemma}
\begin{remark}
We note that we do not have the wellposedness if we remove the condition $f(u)>t$ and $u(t)>0$. In fact, $u=0$ for all $t>0$ is also a solution. The extra condition allows us to obtain the uniqueness for all $t$.
\end{remark}
{\begin{proof}
The condition $f(u)>t$ excludes the possibility of $u(t)=0$ for $t>0$, and thus $\frac{\rd{u}}{\rd{t}}\ne 0$, and one can write $t=t(u)$. Then $t(u)$ satisfies
\begin{equation}
\frac{\rd{t}}{\rd{u}} = \frac{f(u)-t}{u}\,,\quad t(0)=0\,,
\end{equation}
which is a linear ODE, and has the general solution
\begin{equation}
t = \frac{1}{u}\left(\int_0^u f(s)\rd{s}+C\right)\,,
\end{equation}
away from $u=0$. Since $f(s)\sim as^2$ for small $s$, it is clear that $\lim_{u\rightarrow 0}\frac{1}{u}\left(\int_0^u f(s)\rd{s}+C\right)=0$ only holds when $C=0$. This means \eqref{odeu_soln} gives the only solution to \eqref{odeu} satisfying the assumptions.
\end{proof}
}

\begin{lemma}\label{lem:sym}
If $f(u) = f(-u)$ is a symmetric function then $(u,-u)$ solves~\eqref{ode01} if $u$ solves~\eqref{odeu}.
\end{lemma}
The proof is rather straightforward and we omit it.

\begin{lemma}\label{lemcomp}
	Let $(u_1,u_2)$ solves~\eqref{ode01} with initial condition~\eqref{odei}, and $(v,-v)$ solves~\eqref{ode01} with initial condition~\eqref{odei} where $f$ is replaced by an even function $g$ ($g(u)=g(-u)$). Then for small $t$:
	\begin{itemize}
	\item if $f(u)< g(u)$ for all $u$, then $u_1\ge v,u_2\le -v$;
	\item if $f(u)> g(u)$ for all $u$, then $u_1\le v,u_2\ge -v$.
	\end{itemize}
\end{lemma}
\begin{proof}
	Again, we use the shooting method to prove this lemma. We prove the first statement by contradiction. Suppose it is not true, then there exists a $t_0>0$ small enough such that $u_1(t_0)<v(t_0)$ or $u_2(t_0)>-v(t_0)$. Without loss of generality, we assume the former case, and that $-u_2(t_0)\ge u_1(t_0)$. From the previous lemma, there exists a $g$-solution $(v_1,-v_1)$ with $v_1(t_0) > u_1(t_0)$, and this solution hits $g(v_1)-t=0$ line before $t=0$, meaning there is $t_1>0$ such that:
\begin{equation*}
f\left(v_1 - \int_{t_1}^{t_0} F^g_1(v_1,-v_1)\rd{s}\right) = t_1\,.
\end{equation*}
Here $F^g_1(u_1,u_2):=\frac{1}{2}(u_1-u_2)\frac{1}{g(u_1)- t}$ is the forcing term for $v_1$ defined by $g$.

	We claim that there is no $t\le t_0$ such that $u_1(t)\ge v_1(t),u_2(t)\le -v_1(t)$. In fact, this is not true at $t=t_0$. Suppose $t_2$ is the largest time such that this holds, then without loss of generality, we assume that $u_1(t_2)=v_1(t_2)$. Then we have:
	\begin{equation*}
	\begin{split}
	F_1^f(u_1(t_2),u_2(t_2)) & := \frac{1}{2}(u_1(t_2)-u_2(t_2))\frac{1}{f(u_1(t_2))-t_2} \ge \frac{1}{2}(v_1(t_2)+v_1(t_2))\frac{1}{f(v_1(t_2))-t_2} \\
	& > \frac{1}{2}(v_1(t_2)+v_1(t_2))\frac{1}{g(v_1(t_2))-t_2} = F_1^g(v_1(t_2),-v_1(t_2))\,.
	\end{split}
	\end{equation*}
	This contradicts the choice of $t_2$. See Figure \ref{theorem_fig2} (left) for an illustration.
	
	This claim contradicts the fact that $(u_1,u_2)$ can be continued to the time $t_1$, since at this time, $f(u_1)-t < g(u_1)-t \le g(v_1)-t = 0$ or $f(u_2)-t < g(u_2)-t \le g(-v_1)-t = 0$. Thus the first statement is proved. See Figure \ref{theorem_fig2} (right) for an illustration. The second statement can be proved similarly.
\end{proof}
\begin{figure}
\centering
	\includegraphics[width=0.4\textwidth,height=0.25\textheight]{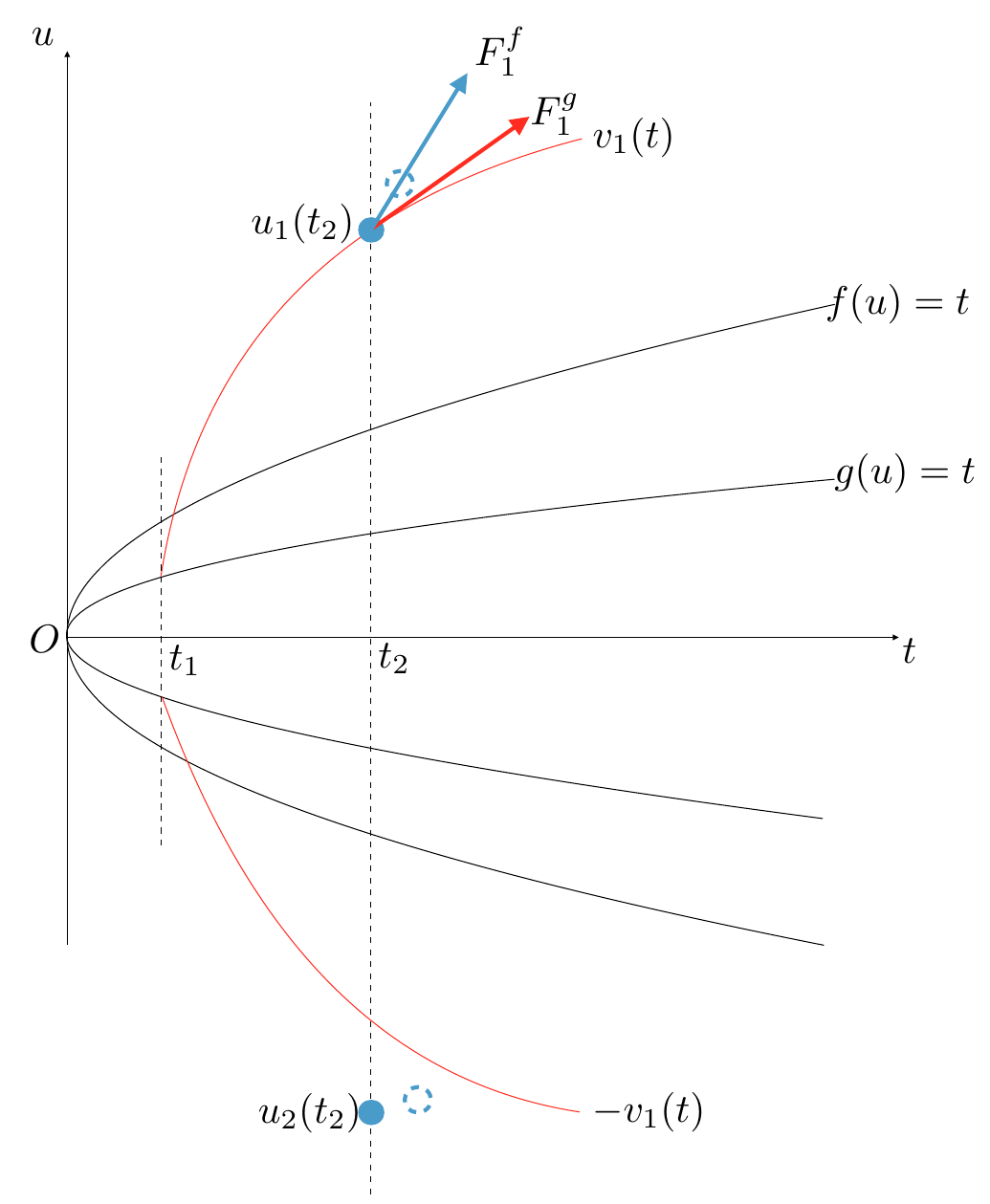}
	\includegraphics[width=0.4\textwidth,height=0.25\textheight]{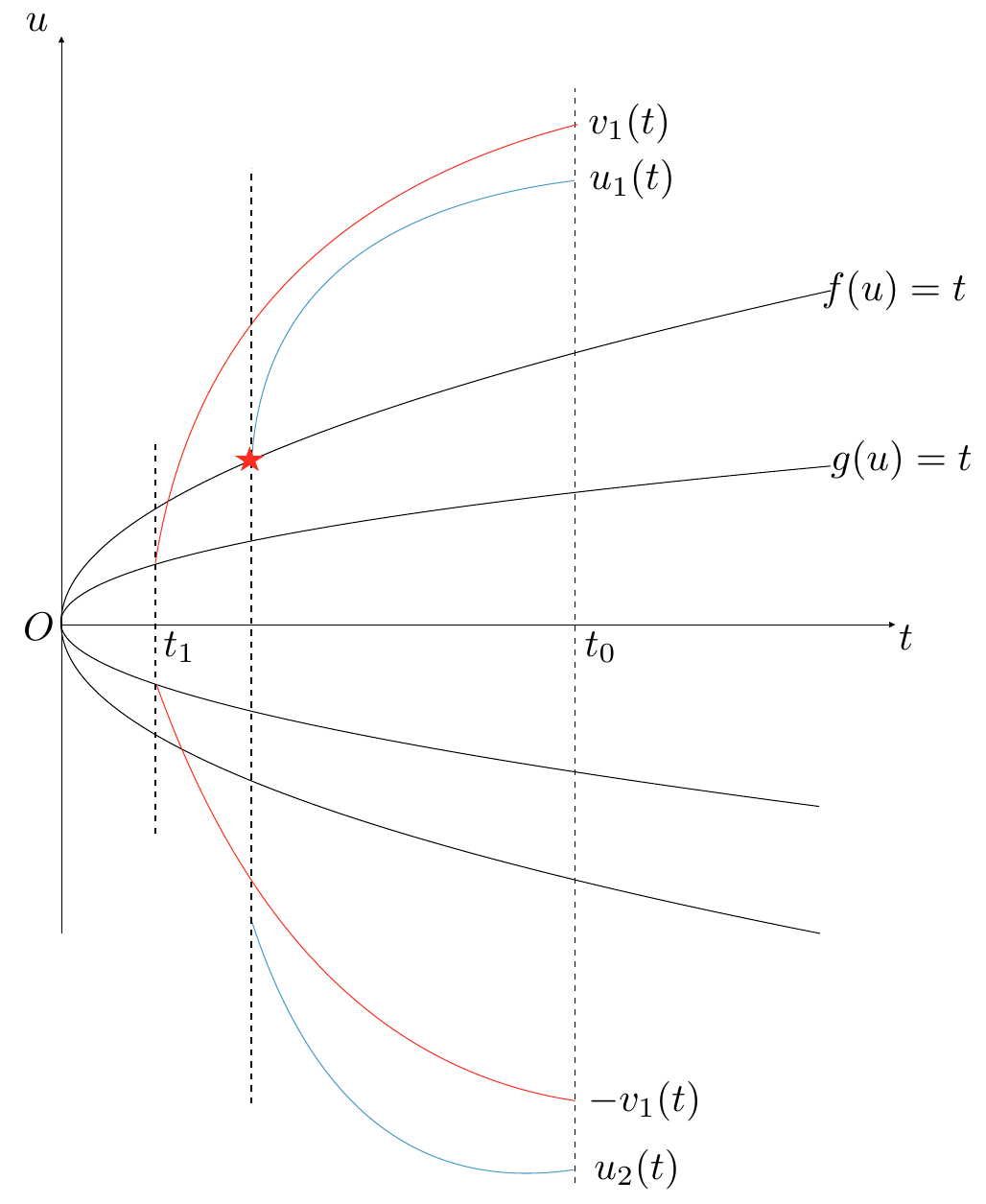}
	\caption{Proof of Lemma \ref{lemcomp}. Left: obtaining a contradiction at time $t_2$. The dashed circle is the approximate position of $u_{1,2}$ at time slightly larger than $t_2$, which contradicts the choice of $t_2$. Right: obtaining the final contradiction. $u_1$ or $u_2$ must touches the curve $f(u)=t$ at some time larger than $t_1$ (the star in the picture).}\label{theorem_fig2}
\end{figure}

We now are ready to show Theorem~\ref{thm1_deterministic}.
\begin{proof}	
Locally at $t=0$, $f(u)\sim au^2$ and thus we approximate~\eqref{ode01} by:
\begin{equation}\label{ode1}
\frac{\rd{u_1}}{\rd{t}} = \frac{1}{2}(u_1-u_2)\frac{1}{au_1^2-t}\,,\quad \frac{\rd{u_2}}{\rd{t}} = -\frac{1}{2}(u_1-u_2)\frac{1}{au_2^2-t}\,,
\end{equation}
which has solution
\begin{equation*}
u_1=-u_2 = (\frac{3}{a}t)^{1/2}\,.
\end{equation*}
To estimate the solution of~\eqref{ode01} near $t=0$, we let $\epsilon<a$, and find $\delta$ such that
	\begin{equation}\label{fu2}
	|f(u)-au^2| < \epsilon u^2 \quad (\forall |u|<\delta)\,,
	\end{equation}
	then according to Lemma \ref{lemcomp}, for small $t$ (small enough so that $u_1^-<\delta$ and $-u_2^-<\delta$)
	\begin{equation}\label{est1}
	u_1^+\le u_1 \le u_1^-, u_2^-\le u_2 \le u_2^+\,,
	\end{equation}
	where $u_i^\pm$ is the solution of (\ref{ode1}) with $a$ replaced by $a\pm \epsilon$, and we conclude the theorem.
\end{proof}

\begin{remark}
We note that the forcing terms in~\eqref{ode01} are Lipschitz away from $f(u) = \alpha t$, and the system is automatically wellposed there. The main difficulty lies in ``small time" regime where $f(u^\ast) = \alpha t^\ast$.
\end{remark}

\section{Smoothness in $z$-space}\label{sec:random}

{As we discussed in the introduction}, there are two sources of uncertainties {in the Burgers equation~\eqref{eq}}: the initial condition $u_\text{in}(x,z)$ and the traveling speed of the wave $\alpha(z)$. In Theorem~\ref{thm_main1} we claim that the physical quantities such as $t^\ast$ (the shock emerging time), $t^\sharp$ (the time for the shock hitting the land), and $x^c$ (the shock location) are smooth functions of $z$, and in Theorem~\ref{thm_tilde1}, we claim that with proper shifting, the solution profile depends on $z$ smoothly as well. These two theorems are proved in Section~\ref{sec:smooth_physical} and Section~\ref{sec:smooth_profile} respectively.

\subsection{Smoothness of physical quantities}\label{sec:smooth_physical}
The main goal in this subsection is to prove Theorem~\ref{thm_main1}, which states that the physical quantities smoothly depend on $z$. We start by proving item (1) of Theorem \ref{thm_main1}:
\begin{proof}[Proof of item (1) of Theorem \ref{thm_main1}]
In the deterministic case, we have shown that 
\begin{equation*}
t^\ast = -\frac{1}{\alpha u_\text{in}'(x^\ast)}\,,
\end{equation*}
and thus to show the regularity of $t^\ast$ on $z$ amounts to showing the regularity of $x^*$ on $z$, since $\alpha$ and $u_{in}$ are assumed to be smooth in $z$. Take the first derivative for example:
\begin{equation}\label{dztast}
\partial_zt^\ast = \frac{1}{(\alpha u_\text{in}'(x^\ast))^2}\partial_z\left(\alpha u_\text{in}'(x^\ast)\right)=\frac{\partial_z\alpha u'_\text{in}(t_\ast)+\alpha\partial_zu'_\text{in}(x^\ast,z)+\alpha u''_\text{in}(x^\ast,z)\partial_zx^\ast}{(\alpha u_\text{in}'(x^\ast))^2}\,.
\end{equation}
$\partial_zu_\text{in}$ and $\partial_z\alpha$ are known to be bounded quantities, and by definition
\[
u_\text{in}''(x^*(z),z) = 0\,,
\]
which gives $|\partial_zt^\ast|<C$, meaning $t^\ast$ is Lipschitz continuous in $z$. Higher derivatives can be analyzed in a similar way except that one also needs to analyze $\partial^k_z x^\ast$. It is a bounded quantity as well and we show it for $k=1$. Since
\begin{equation}
u_\text{in}''(x^*(z),z) = 0\quad\Rightarrow\quad u_\text{in}'''(x^*(z),z)\partial_z x^* + \partial_z u_\text{in}''(x^*(z),z) = 0\,,
\end{equation}
which gives
\begin{equation}\label{dzxst}
\partial_z x^* =-\frac{ \partial_z u_\text{in}''(x^*(z),z)}{u_\text{in}'''(x^*(z),z)}\,.
\end{equation}
\end{proof}

Proving item (2) and (3) in Theorem~\ref{thm_main1} requires more delicate analysis and we leave them to the next subsection. Item (4), however, is a direct corollary of (2).
\begin{proof}[Proof of item (4) of Theorem \ref{thm_main1}, assuming item (2)]
According to the definition:
\begin{equation}\label{tsh}
t^\sharp = \inf \{t:x^c(t)\ge x_0\}\quad\Rightarrow\quad x^c(t^\sharp) = x_0\,,
\end{equation}
meaning $t^\sharp=t^\sharp(x_0)$ is the inverse function of $x^c(t)$ evaluated at $x_0$. According to item (2) in Theorem \ref{thm_main1}, $x^c(t,z)$ is smooth in $z$, then taking $z$-derivative on (\ref{tsh}) gives
\[
\partial_z x^c(t^\sharp,z) + \partial_t x^c(t^\sharp,z)\partial_z t^\sharp = 0\,,
\]
which shows that $|\partial_z t^\sharp|<\infty$ by the assumption that  $\partial_t x^c(t^\sharp,z) \ne 0$. Higher order $z$-derivatives can be handled in the same way.
\end{proof}

We now concentrate on showing items (2) and (3) of Theorem \ref{thm_main1}, which state the smooth dependence of $x^c$ and $u_1-u_2$ on $z$. We divide the proof into two parts: we will first prove the smoothness assuming all the initial shocks are generated at $t^\ast = 0$ and $u^\ast=0$, meaning $u_{1,2}(t=0,z) = 0$ for all $z$; we then shift $(t^\ast,u^\ast)$ to accommodate the general situation stated in Theorem~\ref{thm_main1}. The first part of the proof is summarized in Proposition~\ref{thm:random_u_1} and Proposition~\ref{thm:x_c}, and the second part of the proof follows.

\begin{proposition}\label{thm:random_u_1}
Consider~\eqref{ode01} with initial condition~\eqref{odei} and $\alpha=1$. Suppose the initial profile represented by $-x'_\text{in}(u) = f(u)$ has smooth $z$-dependence, i.e., $f(u;z)\in C^\infty(\mathbb{R}_u,\mathbb{R}_z)$, then for $t$ small enough:
\begin{itemize}
\item[(1)] the $z$-derivatives of $u_1,u_2$ satisfy the estimate
\begin{equation*}
\partial_zu_{1,2} = \mathcal{O}(t^{1/2})\,,
\end{equation*}
\item[(2)] the higher $z$-derivatives of $u_1,u_2$ satisfy:
\begin{equation*}
\partial_z^ku_{1,2} = \mathcal{O}(t^{1/2})\,,
\end{equation*}
\item[(3)] the higher $(z,t)$-derivatives in time satisfy:
\begin{equation*}
\partial_z^k\partial_t^{k'}u_{1,2} = \mathcal{O}(t^{1/2-k'})\,.
\end{equation*}
\end{itemize}
\end{proposition}

\begin{proof}
To obtain the regularity in $z$ direction, one basically needs to take $z$ derivatives and show the bounds. We start with first order derivative of~\eqref{ode01} to show item (1):
\begin{equation*}
\begin{split}
\frac{\rd{\partial_z u_1}}{\rd{t}} & = \frac{1}{2}(\partial_z u_1-\partial_z u_2)\frac{1}{f(u_1)-t} - \frac{1}{2}(u_1-u_2)\frac{1}{(f(u_1)-t)^{2}}(f'(u_1)\partial_z u_1 + \partial_z f(u_1))\,,\\
\frac{\rd{\partial_z u_2}}{\rd{t}} & = -\frac{1}{2}(\partial_z u_1-\partial_z u_2)\frac{1}{f(u_2)-t} + \frac{1}{2}(u_1-u_2)\frac{1}{(f(u_2)-t)^{2}}(f'(u_2)\partial_z u_2 + \partial_z f(u_2))\,.\\
\end{split}
\end{equation*}
In a compact form, it becomes:
\begin{equation}\label{eqn:der_z}
\begin{cases}
\frac{\rd{\partial_z u_1}}{\rd{t}} & = A_{11}\partial_z u_1 + A_{12}\partial_z u_2 + S_1\,,\\
\frac{\rd{\partial_z u_2}}{\rd{t}} & = A_{21}\partial_z u_1 + A_{22}\partial_z u_2 + S_2\,,
\end{cases}
\end{equation}
with initial data
\begin{equation*}
\partial_z u_1(0) = \partial_z u_2(0) = 0\,.
\end{equation*}
Here $A$ terms are the linear terms and $S_{1,2}$ are sources. The terms in $A$ and $S$ can be estimated using the results in Theorem \ref{thm1_deterministic} which states that $u_{1,2}(t)\approx\pm(3a^{-1}t)^{1/2}$, so
\begin{equation*}
A_{11} = \frac{1}{2}\frac{1}{f(u_1)-t}- \frac{1}{2}(u_1-u_2)\frac{1}{(f(u_1)-t)^{2}}f'(u_1) \approx -\frac{5}{4t}\,,
\end{equation*} 
and similarly:
\begin{equation*}
A_{12} = -\frac{1}{2}\frac{1}{f(u_1)-t} \approx -\frac{1}{4t}\,,\quad A_{21}\approx -\frac{1}{4t},\quad A_{22}\approx -\frac{5}{4t}\,.
\end{equation*}
Here we use the notation $A(t)\approx B(t)$ to mean $\lim_{t\rightarrow0+}\frac{A(t)}{B(t)}=1$. Noting that $f(u;z)\sim a(z)u^2$ (see (\ref{au2})), $\partial_zf\sim \partial_zau^2$, and thus
\[
S_1 =- \frac{1}{2}(u_1-u_2)\frac{1}{(f(u_1)-t)^{2}}\partial_zf(u_1) =\mathcal{O}(t^{-1/2}),\quad S_2=\mathcal{O}(t^{-1/2})\,.
\]

Then we perform the standard energy estimate of $L^2$ type for~\eqref{eqn:der_z} by multiplying it on both sides with $\partial_zu_{1,2}$ to have:
\begin{equation*}
\begin{split}
\frac{1}{2}\frac{\rd}{\rd{t}}(\partial_zu_1)^2 & \le -(\frac{5}{4}-\epsilon)\frac{1}{t} (\partial_zu_1)^2 + (\frac{1}{4}+\epsilon)\frac{1}{t}|\partial_zu_1\partial_zu_2| + Ct^{-1/2}|\partial_zu_1|\,, \\
\frac{1}{2}\frac{\rd}{\rd{t}}(\partial_zu_2)^2 & \le -(\frac{5}{4}-\epsilon)\frac{1}{t} (\partial_zu_2)^2 + (\frac{1}{4}+\epsilon)\frac{1}{t}|\partial_zu_1\partial_zu_2| + Ct^{-1/2}|\partial_zu_2|\,. \\
\end{split}
\end{equation*}
Add the two inequalities and use the fact that 
\begin{equation*}
|\partial_zu_1\partial_zu_2| \le \frac{1}{2}((\partial_zu_1)^2 + (\partial_zu_2)^2)\,,\quad\text{and}\quad t^{-1/2}|\partial_zu_1| \le \epsilon_1\frac{1}{t}(\partial_zu_1)^2 + \frac{1}{4\epsilon_1},\quad \forall \epsilon_1 >0\,,
\end{equation*}
one gets:
\begin{equation*}
\frac{1}{2}\frac{\rd}{\rd{t}}((\partial_zu_1)^2+(\partial_zu_2)^2) \le -\left(\frac{5}{4}-\epsilon-(\frac{1}{4}+\epsilon) -C\epsilon_1\right)\frac{1}{t}((\partial_zu_1)^2+(\partial_zu_2)^2) + \frac{C}{2\epsilon_1}\,.
\end{equation*}
Choosing $\epsilon=1/4,\quad \epsilon_1 = 1/(2C)$, one gets
\begin{equation*}
\frac{\rd}{\rd{t}}((\partial_zu_1)^2+(\partial_zu_2)^2) \le 2C^2\,,
\end{equation*}
which finishes the proof of item (1).

Extending it to higher derivatives requires mathematical induction. We assume $\partial_z^ju_{1,2}=\mathcal{O}(t^{1/2})$ holds true for all $j<k$, and we now show it for the $k$-th derivative as well. Taking the $k$-th derivative on $z$ we have:
\begin{equation}\label{eqn:der_zk}
\begin{cases}
\frac{\rd{\partial_z^k u_1}}{\rd{t}} & = A_{11}\partial_z^k u_1 + A_{12}\partial_z^k u_2 + S_1^k\,,\\
\frac{\rd{\partial_z^k u_2}}{\rd{t}} & = A_{21}\partial_z^k u_1 + A_{22}\partial_z^k u_2 + S_2^k\,,
\end{cases}
\end{equation}
where $A_{mn}$ has the same definition as in~\eqref{eqn:der_z} ($m,n=1,2$). The source term $S_1^k$ is, however, much more complicated:
\begin{equation}\label{eqn:S_term}
S_1^k=\sum \frac{c\partial_z^{r_1}(u_1-u_2)}{(f(u_1)-t)^{1+r_2}}\prod_{j=1}^{r_2}\left( \partial_z^{r_{3,j}}\partial_u^{r_{4,j}}f(u_1)\prod_{l=1}^{r_{4,j}}\partial_z^{r_{5,j,l}}u_1\right)\,,
\end{equation}
where $c$ is a constant depending on the summation indices, and the indices in the summation satisfy the relation
\begin{equation*}
r_1+\sum_{j=1}^{r_2}(r_{3,j}+\sum_{l=1}^{r_{4,j}}r_{5,j,l}) = k\,,\quad r_1\le k-1\,,\quad r_{5,j,l}\le k-1\,.
\end{equation*}
Noting that
\begin{equation*}
f(u_1) \sim u_1^2 = \mathcal{O}(t),\quad f'(u_1)\sim u_1=\mathcal{O}(t^{1/2}),\quad\text{and}\quad \partial_u^rf(u_1) = \mathcal{O}(1),\,\,r\ge 2\,,
\end{equation*}
and as a result, $\partial_u^rf(u_1) \lesssim \mathcal{O}(t^{1-r/2}),\,\,r\ge 0$. In view of $\partial_z^lu_1 = \mathcal{O}(t^{1/2})$ for $l\le k-1$, the order of the term~\eqref{eqn:S_term} is (in term of the power of $t$)
\begin{equation*}
\frac{1}{2} - (1+r_2) + \sum_{j=1}^{r_2} ((1-\frac{r_{4,j}}{2}) + \frac{r_{4,j}}{2}) = -\frac{1}{2}\,.
\end{equation*}
The term $S_2^k$ can be analyzed in the same way. Using the energy estimate we conclude with the result. Item (3) is obtained using the induction argument as well, and we leave the proof to appendix.
\end{proof}

\begin{proposition}\label{thm:x_c}
With the same assumptions as in Proposition~\ref{thm:random_u_1}, one has
\begin{equation*}
\partial_z^k\partial_t^{k'}x^c = \chi_{(k'=0)}\partial_z^k x^*(z) + \mathcal{O}(t^{3/2-k'})\,.
\end{equation*}
\end{proposition}
\begin{proof}
It follows easily from checking~\eqref{eqn:center_location} which holds true on $t>t^\ast=0$ with $x^c(0) = x^\ast$. Integrating (\ref{eqn:center_location}) in $t$ we get
\begin{equation}
x^c(t,z) = x^*(z) + \int_0^t \frac{u_1(s,z)+u_2(s,z)}{2}\rd{s}\,.
\end{equation}
Taking its $(z,t)$-derivative, we have
\begin{equation}
\partial_z^kx^c(t,z) = \partial_z^kx^*(z) + \int_0^t \frac{\partial_z^ku_1(s,z)+\partial_z^ku_2(s,z)}{2}\rd{s}\,,
\end{equation}
when $k'=0$ and
\begin{equation}
\partial_z^k\partial_t^{k'} x^c=  \partial_t^{k'}\left(\int_0^t \frac{\partial_z^ku_1(s,z)+\partial_z^ku_2(s,z)}{2}\rd{s} \right) = \frac{\partial_z^k\partial_t^{k'-1}u_1(t,z)+\partial_z^k\partial_t^{k'-1}u_2(t,z)}{2} \,,
\end{equation}
when $k'>0$. These combined with Proposition \ref{thm:random_u_1} give the conclusion.
\end{proof}

With these preparations for equations with special initial data $u^\ast = t^\ast = 0$, we are ready to perform shifting for proving Theorem~\ref{thm_main1}.

\begin{proof}[Proof of items (2) and (3) of Theorem~\ref{thm_main1}]

We translate $u_1,u_2$ to enforce the initial condition (\ref{odei}) by defining $\bar{u}_1,\bar{u}_2$ as
\begin{equation}\label{eqn:shifting_t}
\bar{u}_{1,2}(t,z) = u_{1,2}\left(\frac{1}{\alpha(z)}t+t^\ast(z),z\right)-u^*(z)\,.
\end{equation} 
$\bar{u}_{1,2}(t,z)$ then satisfies the same system (\ref{ode01}) with initial condition (\ref{odei}) and $\alpha=1$. $f$, however, is also shifted:
\begin{equation}
\bar{f}(\bar{u},z) = f(\bar{u}+u^*(z)) - \alpha(z)t^*(z)\,.
\end{equation}
It is clear that $\bar{f}$ is smooth and satisfies Assumption \ref{ass:x_u_initial2}. By the assumption that there exists $\delta>0$ such that $u_++\delta \le u^\ast(z) \le u_--\delta$ for all $z$, we have $\bar{f}$ well-defined for $u\in[-\delta,\delta]$.

According to Proposition~\ref{thm:random_u_1}:
\begin{equation}
\partial_z^k\partial_t^{k'}\bar{u}_{1,2}=\mathcal{O}(t^{1/2-k'})\,,
\end{equation}
and considering $u_{1,2}(t,z) = \bar{u}_{1,2}(\alpha(z)(t-t^*(z)),z)+u^*(z)$, and taking the smoothness of $\alpha(z)$, $t^*(z)$ and $u^*(z)$ into account, we obtain the estimate
\begin{equation}\label{dku12}
\partial_z^k u_{1,2} =  \mathcal{O}((t-t^*)^{1/2-k})\,,
\end{equation}
which implies item (3) of Theorem~\ref{thm_main1}. To estimate $x^c$, we change $u_{1,2}$ to $\bar{u}_{1,2}$ as well:
\begin{equation*}
\begin{split}
x^c(t,z) = & x^*(z) + \alpha(z)\int_{t^*(z)}^t \frac{u_1(s,z)+u_2(s,z)}{2}\rd{s} \\ 
= & x^*(z) + \alpha(z)\int_{t^*(z)}^t \frac{\bar{u}_{1}(\alpha(s-t^*(z)),z)+\bar{u}_{2}(\alpha(s-t^*(z)),z))}{2}+u^*(z)\rd{s} \\
= & x^*(z) + \frac{1}{2}\int_0^{\alpha(z)(t-t^*(z))}\left[\bar{u}_1(s,z)+\bar{u}_2(s,z)\right]\rd{s} + \alpha(z)u^*(z)(t-t^*(z)) \\
= &\bar{x}^c(\alpha(z)(t-t^*(z)),z) + \alpha(z)u^*(z)(t-t^*(z)) \,.
\end{split}
\end{equation*}
Take its $z$ derivative up to order $k$ on both sides and use Proposition \ref{thm:x_c} together with the smoothness of  $\alpha(z)$, $t^*(z)$ and $u^*(z)$, we conclude the theorem:
\begin{equation*}
\partial_z^k x^c(t,z) = \partial_z^k x^*(z) + \mathcal{O}(t^{3/2-k}) + \mathcal{O}(1) = \mathcal{O}(t^{\min\{3/2-k,0\}})\,.
\end{equation*}
 
\end{proof}

\subsection{Smoothness of the shifted solution profile}\label{sec:smooth_profile}
This is to mathematically justify Theorem \ref{thm_tilde1} which guarantees the $z$-regularity of the shifted solution $\tilde{u}$.
\begin{proof}[Proof of Theorem \ref{thm_tilde1}]
Recall $\tilde{u}$ in~\eqref{tildeu} and we take its $k$-th derivative in $z$:
\begin{equation}\label{dzkut}
\partial_z^k \tilde{u} = \sum c\partial_t^{r_1}\partial_x^{r_2}\partial_z^{r_3}u \prod_{j=1}^{r_2} \partial_z^{r_{4,j}}t^*(z) \prod_{j=1}^{r_2}\partial_z^{r_{5,j}}(x^c(t+t^*(z),z))\,,
\end{equation}
where the indices satisfy
\begin{equation}
r_3 + \sum_{j=1}^{r_1}r_{4,j} + \sum_{j=1}^{r_2}r_{5,j} = k\,.
\end{equation}
Then we further expand
\begin{equation}\label{eqn:dz_xc}
\partial_z^{r_{5,j}}(x^c(t+t^*(z),z)) = \sum c \partial_t^{r_{6,j}}\partial_z^{r_{7,j}}x^c\prod_{l=1}^{r_{6,j}}\partial_z^{r_{8,j,l}}t^*(z)\,,\quad\text{with}\quad r_{7,j}+\sum_{l=1}^{r_{6,j}}r_{8,j,l} = r_{5,j}\,.
\end{equation}

The second factor in (\ref{dzkut}) is $\partial_z^{r_{4,j}}t^*(z)$ which is of $\mathcal{O}(1)$. To deal with the third factor of (\ref{dzkut}), we use~\eqref{eqn:dz_xc} and only evaluate $\partial_t^{r_{6,j}}\partial_z^{r_{7,j}}x^c$, with $r_{6,j}+r_{7,j}\le r_{5,j}$. According to Proposition \ref{thm:x_c} we have
\begin{equation}\label{eqn:dtzxc}
\partial_t^{r_{6,j}}\partial_z^{r_{7,j}}x^c(t+t^*(z),z)=\mathcal{O}(t^{\min\{3/2-r_{6,j},0\}})\,.
\end{equation}

The first factor $\partial_t^{r_1}\partial_x^{r_2}\partial_z^{r_3}u$ is more complicated. To do that we take the derivative $\partial_t^{r_1}\partial_x^{r_2}\partial_z^{r_3}$ of
$x(t,u(t,x,z),z)=x$ for:
\begin{equation*}
\partial_u x \partial_t^{r_1}\partial_x^{r_2}\partial_z^{r_3}u + \sum c \partial_t^{r_1'}\partial_u^{r_2'}\partial_z^{r_3'}x \prod_{j=1}^{r_2'}\partial_t^{r_{4,j}'}\partial_x^{r_{5,j}'}\partial_z^{r_{6,j}'}u = \chi_{(r_1=r_3=0,r_2=1)}\,,
\end{equation*}
with
\begin{equation*}
r_1'+\sum_{j=1}^{r_2'}r_{4,j}' = r_1,\quad r_3'+\sum_{j=1}^{r_2'}r_{6,j}' = r_3,\quad \sum_{j=1}^{r_2'}r_{5,j}' = r_2\,.
\end{equation*}
Since $x(t,u,z) = x_\text{in}(u,z) + ut$ is away from the shock, all the derivatives of $x$ are of $\mathcal{O}(1)$. Thus one can show by induction on $r_1+r_2+r_3$ that
\begin{equation}\label{eqn:dtxzu}
\partial_t^{r_1}\partial_x^{r_2}\partial_z^{r_3}u = \mathcal{O}(|\partial_x u|^{2(r_1+r_2+r_3)-1})\,,
\end{equation}
where $|\partial_x u|$ is supposed to be large near shock. We then claim that 
\begin{equation}\label{eqn:claim_pu_x}
\left|\partial_x u(t+t^*(z),x+x^c(t+t^*(z),z),z)\right|\le \frac{2}{|x|}\,.
\end{equation}
In fact, suppose $x>0$, then
\begin{equation*}
u(t+t^*,x+x^c,z)-u(t+t^*,x^c,z) = \int_{x^c}^{x+x^c}\partial_x u(t+t^*,y,z)\rd{y} \le x \partial_x u(t+t^*,x+x^c,z)\,,
\end{equation*}
where the inequality holds because $u(t,x,z)$ is convex in $x$ for $x>x^c(t,z)$, and thus $\partial_x u(t,y,z) \le \partial_x u(t,x+x^c,z)$. Taking absolute value and using $\partial_x u<0$, we get
\begin{equation*}
\begin{split}
|x\partial_x u(t+t^*,x+x^c,z)| \le & |u(t+t^*,x+x^c,z)-u(t+t^*,x^c,z)| \\ \le & |u(t+t^*,x+x^c,z)|+|u(t+t^*,x^c,z)| \le 2\,,
\end{split}
\end{equation*}
which leads to~\eqref{eqn:claim_pu_x}. The case $x<0$ is similar.

In conclusion, plugging~\eqref{eqn:dtxzu} and~\eqref{eqn:dtzxc} into~\eqref{dzkut}, and using the fact that $r_1+r_2+r_3\le k$ and $\sum_{j=1}^{r_2}r_{6,j}\le \sum_{j=1}^{r_2}r_{5,j} \le k$, we conclude the theorem.
\end{proof}

Theorem \ref{thm_interp0} immediately follows from the following proposition. It is a standard result from approximation theory and we leave the proof to the appendix.
\begin{proposition}\label{thm_interp}
Let $f=f(z)\in C^{m+1}(-1,1)$. Then the polynomial interpolation (\ref{poly_interp}) has $m$-th order accuracy:
\begin{equation}\label{thm_interp_1}
|f(z)-f^N(z)| \le \frac{C(m)\|\partial_z^{m+1}f\|_{L^\infty}}{N^m},\quad \forall z\in [-1,1],
\end{equation}
for $N\ge 2m$. Furthermore, if $\pi(z)$ is supported on $[-1,1]$, then we have the error estimate
\begin{equation}\label{thm_interp_2}
|\mathbb{E}(f)-\mathbb{E}(f^N)| \le \frac{C(m)\|\partial_z^{m+1}f\|_{L^\infty}}{N^m}\,,
\end{equation}
\begin{equation}\label{thm_interp_3}
 |\textnormal{var}(f)-\textnormal{var}(f^N)| \le \frac{C(m)(\min\{\|f-f^N\|_{L^\infty},N^2\|f\|_{L^\infty}\}+\|f\|_{L^\infty})\|\partial_z^{m+1}f\|_{L^\infty}}{N^m}\,.
\end{equation}
\end{proposition}


\section{General scalar conservation laws with convex fluxes}\label{sec:general}
All the results for the Burgers' equation can be extended to study general scalar conservation laws with convex flux term. The proof itself is tedious but contains little novelty and thus we only outline the strategies. In the general cases, the equation writes:
\begin{equation}\label{eqF}
\begin{cases}
\partial_t u + \partial_x F(u) = 0\,,\\
\lim_{x\rightarrow \pm\infty}u_\text{in}(x) = \mp 1\,.
\end{cases}
\end{equation}
where the flux function $F$, deterministic, is smooth and strictly convex. We also assume that the initial data is decreasing, and therefore the inverse function $x(u)$ is well-defined on $(-1,1)$. Note that the domain can be generalized to treat $(u_+, u_-)$. The derivation is the same: one flips $x-u$ coordinates and derives the equation for $x(u)$ as a function of $u$. The reformulation allows us to obtain an explicit expression for the dynamics of the physical quantities such as $t^\ast$, $t^\sharp$ and $x^c$. As done for the Burgers' equation we first reformulate the equation, obtain the ODE system, and study its dependence on the unknown variable $z$.

\subsection{Reformulation of the equation}
\subsubsection{Before shocking emergence}
As done for the Burgers' equation, upon flipping $x$ and $u$, one writes the dynamics of $x(u)$ as:
\begin{equation}\label{eq1F}
\begin{cases}
\partial_tx(t,u) = F'(u)\,,\quad u\in(-1,1)\,,\\
x_\text{in}(u) = x(t=0,u)\,.
\end{cases}
\end{equation}
Convex $F$ gives the increasing $F'$. We denote the inverse function $G$:
\begin{equation}
G(F'(u)) = F'(G(u)) = u\,.
\end{equation}
Plugging it back into~\eqref{eq1F} and denoting $y(t,u) = x(t,G(u))$, we have:
\begin{equation*}
\begin{cases}
\partial_ty(t,u) = \partial_tx(t,G(u)) = u\,,\quad u\in(-1,1)\,,\\
y_\text{in}(u) = x(t=0,G(u)) = x_\text{in}(G(u))\,.
\end{cases}
\end{equation*}
As was done for the Burgers' equation, we take one more derivative on $u$ and obtain:
\begin{equation*}
\begin{cases}
\partial_t\partial_uy(t,u)  = 1\,,\quad u\in(-1,1)\,,\\
y'_\text{in}(u) = x'_\text{in}(G(u))G'(u)\,.
\end{cases}
\end{equation*}
Therefore
\begin{equation*}
\partial_uy = y'_\text{in}(u)+t\,.
\end{equation*}
and equivalently, the earliest shock appears at $t^\ast = -\min y'_\text{in}(u)$ and we assume there is one and only one and set it as:
\begin{equation*}
t^\ast = -\min y'_\text{in}(u) = -y'_\text{in}(F'(u^\ast))\,.
\end{equation*}

\subsubsection{After the emergence of the shock}
Once the shock appears, on the $u-x$ plane, a ``flat" region appears. We denote $u_1$ and $u_2$ the top and the bottom of the shock point, then between $(u_2,u_1)$ the solution is a constant, which moves horizontally with speed:
\begin{equation*}
s = \frac{F(u_1)-F(u_2)}{u_1-u_2}\,,
\end{equation*}
meaning
\begin{equation}\label{eqn:center_locationF}
\frac{\rd}{\rd t}x^c = \frac{F(u_1)-F(u_2)}{u_1-u_2}\,,\quad\text{with}\quad x^c(t^\ast) = x^\ast\,.
\end{equation}
where $x^c$ denotes the shock location. With the same derivation as in Section~\ref{sec:reformulation}, one has:
\begin{equation}\label{ode01F}
\begin{cases}
\frac{\rd{u_1}}{\rd{t}} & = F_1(u_1,u_2) = \left(F'(u_1)-\frac{F(u_1)-F(u_2)}{u_1-u_2}\right)(f(u_1)-F''(u_1)t)^{-1}\,, \\
\frac{\rd{u_2}}{\rd{t}} & = F_2(u_1,u_2) = -\left(\frac{F(u_1)-F(u_2)}{u_1-u_2}-F'(u_2)\right)(f(u_2)-F''(u_2)t)^{-1}\,,
\end{cases}
\end{equation}
with initial condition $u_1(t^\ast) = u_2(t^\ast) = u^\ast$. Here we denoted $f(u) = -x'_\text{in}(u)$. Considering $F'$ is an increasing function, we see that
\begin{equation*}
\frac{\rd u_1}{\rd t} > 0 > \frac{\rd u_2}{\rd t}\,.
\end{equation*}

\subsubsection{Summary}
To summarize the reformulation, in the general convex flux case, when writes on $x(u)$ plane, $x$ satisfies equation
\begin{equation}\label{eqn:summary_x_u_F1}
\begin{cases}
t<t^\ast = -y'_\text{in}(u^\ast):\quad&\text{Equation~\eqref{eq1F}}\,,\\
t>t^\ast:\quad&
\begin{cases}
\text{Equation~\eqref{eq1F}}\quad \text{with}\quad u\in(-1,u_2)\cup(u_1,1)\,,\\
\text{Equation~\eqref{eqn:center_locationF}}\quad \text{with}\quad u\in (u_2,u_1)\,,
\end{cases}
\end{cases}
\end{equation}
with $u_1$ and $u_2$ being the shock locations satisfying the ODE system~\eqref{ode01F}.

\subsection{Shock behavior in small time (general flux)}
Assume a shock emerges at $t^\ast = 0,\,u^\ast = 0$, then to understand the short time behavior of the ODE system is equivalent to understanding the forcing terms in~\eqref{ode01F}. Near $u_{1,2}=0$ we can approximate:
\begin{equation}
\begin{split}
&F'(u_1)-\frac{F(u_1)-F(u_2)}{u_1-u_2}\\
= & F'(0)+F''(0)u_1 - \frac{F'(0)(u_1-u_2) + \frac{1}{2}F''(0)(u_1^2-u_2^2))}{u_1-u_2} +\mathcal{O}(u_1^2,u_2^2,u_1u_2)\\
= & \frac{1}{2}F''(0)(u_1-u_2) +\mathcal{O}(u_1^2,u_2^2,u_1u_2)\,,
\end{split}\,,
\end{equation}
and thus in the leading order:
\begin{equation}
\begin{cases}
\frac{\rd{u_1}}{\rd{t}} & = \frac{1}{2}F''(0)(u_1-u_2)(au_1^2-F''(0)t)^{-1}\,, \\
\frac{\rd{u_2}}{\rd{t}} &  = -\frac{1}{2}F''(0)(u_1-u_2)(au_2^2-F''(0)t)^{-1}\,,
\end{cases}
\end{equation}
where $a=\frac{a_1F''(0)^2}{G'(F'(0))}$ is a positive number. For small time, the solution is explicit:
\begin{equation}
u_1=-u_2 = (ct)^{1/2},\quad c = \frac{3F''(0)}{a}\,.
\end{equation}

\subsection{Regularities in the random space}
Studying the solution's regularity in the random space is the same as the analysis carried out in Section~\ref{sec:random}. Due to the complexity of the formula, we only present the first derivative in $z$ of~\eqref{ode01F}. One takes the first derivation of~\eqref{ode01F}:
\begin{equation}\label{ode01Fz}
\begin{split}
\frac{\rd{\partial_z u_1}}{\rd{t}} = & \left[F''_1\partial_z u_1-\frac{F'_1\partial_z u_1-F'_2\partial_z u_2}{u_1-u_2} + \frac{(F_1-F_2)(\partial_z u_1 - \partial_z u_2)}{(u_1-u_2)^2}\right](f_1-F''_1t)^{-1} \\
& - \left[F'_1-\frac{F_1-F_2}{u_1-u_2}\right](f_1-F''_1t)^{-2}(f'_1\partial_z u_1 - F'''_1\partial_z u_1 t + \partial_z f_1) \,,\\
\frac{\rd{\partial_z u_2}}{\rd{t}}  = & -\left[\frac{F'_1\partial_z u_1-F'_2\partial_z u_2}{u_1-u_2}- \frac{(F_1-F_2)(\partial_z u_1 - \partial_z u_2)}{(u_1-u_2)^2} - F''_1\partial_z u_1\right](f_2-F''_2t)^{-1} \\
& + \left[\frac{F_1-F_2}{u_1-u_2}-F'_2\right](f_2-F''_2t)^{-2}(f'_2\partial_z u_2 - F'''_2\partial_z u_2 t + \partial_z f_2) \,,\\
\end{split}
\end{equation}
where we have used $\gamma_{1,2}$ to denote $\gamma(u_1)$ or $\gamma(u_2)$ respectively for all quantities. In a compact form, it writes:
\begin{equation*}
\begin{cases}
\frac{\rd{\partial_z u_1}}{\rd{t}} & = A_{11}\partial_z u_1 + A_{12}\partial_z u_2 + S_1\,,\\
\frac{\rd{\partial_z u_2}}{\rd{t}} & = A_{21}\partial_z u_1 + A_{22}\partial_z u_2 + S_2\,,\\
\end{cases}\quad\text{with}\quad \partial_zu_{1,2}(0) = 0\,.
\end{equation*}
In the equation,
\begin{equation*}
\begin{split}
A_{11} = &  \left[F''_1 - \frac{F'_1}{u_1-u_2} + \frac{F_1-F_2}{(u_1-u_2)^2}\right](f_1-F''_1t)^{-1}  - \left[F'_1-\frac{F_1-F_2}{u_1-u_2}\right](f_1-F''_1t)^{-2}(f'_1 - F'''_1t)\,, \\
A_{12} = & \left[\frac{F'_2}{u_1-u_2} - \frac{F_1-F_2}{(u_1-u_2)^2}\right](f_1-F''_1t)^{-1}\,,\\
S_1 = &- \left[F'_1-\frac{F_1-F_2}{u_1-u_2}\right](f_1-F''_1t)^{-2}\partial_z f_1\,.
\end{split}
\end{equation*}
To analyze the term $A_{11}$, we note that
\begin{equation}
\begin{split}
& F''_1\approx F''(0)\,, \quad -\frac{1}{2}F''(0) \approx - \frac{F'_1}{u_1-u_2} + \frac{F_1-F_2}{(u_1-u_2)^2}\,,\\ 
& f(u_1)-F''(u_1)t \approx (ac-F''(0))t = 2F''(0)t \,,\\
& f'(u_1)-F'''(u_1)t \approx 2a(ct)^{1/2}\,,
\end{split}
\end{equation}
which allows us to bound
\begin{equation*}
A_{11}\approx \frac{1}{2}F''(0)\cdot \left(2F''(0)t\right)^{-1} - 2actF''(0)(2F''(0)t)^{-2} = -\frac{5}{4}t^{-1}\,,
\end{equation*}
Similarly one has:
\begin{equation*}
A_{22} \approx -\frac{1}{4}t^{-1}\,,\quad\text{and}\quad S_1 = \mathcal{O}(t^{-1/2})\,.
\end{equation*}
All together,
\begin{equation*}
\frac{\rd}{\rd t}\left[(\partial_zu_1)^2+(\partial_zu_2)^2\right] \leq C\,,
\end{equation*}
and $H_1(\rd z)$ norm of $u_{1,2}$ grows no more than a rate of $\mathcal{O}(t^{1/2})$. 

\section{Conclusion}
Uncertainty quantification for hyperbolic conservation laws is considered a very challenging task due to the intrinsic discontinuities in the solution in both physical and random spaces. Such discontinuities in the solution profile prevent the generalized polynomial chaos type methods to be effective. We give a counter-argument in this paper, and we demonstrate, under some mild assumptions on the initial condition, that:
\begin{itemize}
\item[1.] {there exists physical observables depending smoothly on external randomness};
\item[2.] with proper shifts of the solution in time and space, the entire solution profile also smoothly depends on the external randomness.
\end{itemize}
We have to emphasize that the main goal of the paper is not to justify the gPC method's use on hyperbolic systems, but rather, to provide a new perspective: for wave-like equations with randomness, solution profile may not be the right ``quantity of interests" to evaluate, and a slight change (the proper shifts) could regularize the problem significantly.

\newpage

\appendix
\section{Supplementary Proofs}\label{sec:appendixA}
For the completeness of the paper we include the proofs with tedious calculation here.
\subsection{Wellposedness of the ODE system~\eqref{ode01}}\label{sec:appendix_ode}
We show the wellposedness of the ODE system~\eqref{ode01}. In fact, the two forcing terms $F_1$ and $F_2$ in~\eqref{ode01} are Lipschitz continuous on $u_1$ and $u_2$ if $f(u_{1,2})-\alpha t$ are away from $0$, and the lemma below shows that they keep being Lipschitz as long as $f(u_{1,2})-\alpha t>0$:
\begin{lemma}
Assume $u_{1,2}(t)$ solves (\ref{ode01}) with $f(u_{1,2}(t_1))-\alpha t_1>0$ for some $t_1$. Then there exists $c>0$ such that $f(u_{1,2}(t))-\alpha t>c$ for all $t>t_1$.
\end{lemma}
\begin{proof}
Using (\ref{ode01}), we obtain
\begin{equation}\label{fut}
\frac{\rd}{\rd{t}}(f(u_1)-\alpha t)=f'(u_1)\frac{\rd}{\rd{t}}u_1 - \alpha =\frac{\alpha }{2}(u_1-u_2)f'(u_1)\frac{1}{f(u_1)-\alpha t}-\alpha \,.
\end{equation}

Since $u_1$ is increasing, $u_2$ is decreasing, and $f'(u)>0$ is increasing for $u>u^\ast$, one has
\begin{equation}
\frac{1}{2}(u_1-u_2)f'(u_1) \ge [\frac{1}{2}(u_1-u_2)f'(u_1)]|_{t=t_1}=:c_1>0,\quad \forall t\ge t_1\,.
\end{equation}
According to the ODE~\eqref{fut},
\begin{equation}\label{fut1}
f(u_1)-\alpha t \ge \min\{c_1/2,(f(u_1)-\alpha t)_{t=t_1}\}=:c>0,\quad \forall t\ge t_1\,.
\end{equation}
In fact, if at any time $t$ one has $0<f(u_1)-\alpha t<\frac{2}{3}c_1$, then one has $\frac{\rd}{\rd{t}}(f(u_1)-\alpha t) \ge \alpha (c_1\frac{3/2}{c_1}-1) = \alpha /2$. Therefore starting from $t=t_1$, $f(u_1)-\alpha t$ keeps increasing unless it becomes larger than $\frac{2}{3}c_1$. This implies (\ref{fut1}).

The proof for $f(u_2)-t$ is similar.
\end{proof}


\subsection{Proof for item (3) in Proposition~\ref{thm:random_u_1}}
The proof is moved here merely because of the highly involved calculation. The idea still follows that for the rest of the proposition.
\begin{proof}
We use induction on $(k,k')$. Since we already have the cases $(k,0)$ (Proposition \ref{thm:random_u_1}), we may assume that all cases $(j,j')$ with $j<k$ and $j=k,\,j'\le k'$ are already proved, and then prove the case $(k,k'+1)$. Taking $k'$-th $t$-derivative of (\ref{eqn:der_zk}) gives
\begin{equation}\label{eqn:der_zk1}
\begin{split}
\partial_z^k \partial_t^{k'+1}u_1 & =  S_1^{k,k'} := \partial_t^{k'}(A_{11}\partial_z^k u_1 + A_{12}\partial_z^k u_2 + S_1^k)\,, \\
\partial_z^k \partial_t^{k'+1}u_2 & =  S_2^{k,k'}  := \partial_t^{k'}(A_{21}\partial_z^k u_1 + A_{22}\partial_z^k u_2 + S_2^k)\,.\\
\end{split}
\end{equation}
Notice that every term in $(A_{11}\partial_z^k u_1 + A_{12}\partial_z^k u_2 + S_1^k)$ is of the form (\ref{eqn:S_term}), with possibly  $r_1=k$ or $r_{5,j,l}=k$. Taking $\partial_t^{k'}$ of (\ref{eqn:S_term}) gives terms of the form
\begin{equation}\label{eqn:S_term1}
\frac{c\partial_z^{r_1}\partial_t^{r_1'}(u_1-u_2)}{(f(u_1)-t)^{1+r_2+r_2'}}\prod_{j=1}^{r_2'}\partial_t^{r_{3,j}'}(f(u_1)-t)   \prod_{j=1}^{r_2}\left( \partial_z^{r_{3,j}}\partial_u^{r_{4,j}+r_{4,j}'}f(u_1)\prod_{l=1}^{r_{4,j}'}\partial_t^{r_{5,j,l}'}u_1\prod_{l=1}^{r_{4,j}}\partial_z^{r_{5,j,l}}\partial_t^{r_{6,j,l}'}u_1\right)\,,
\end{equation}
with
\begin{equation}
r_1' + \sum_{j=1}^{r_2'}r_{3,j}' + \sum_{j=1}^{r_2}\left(\sum_{l=1}^{r_{4,j}'}r_{5,j,l}' + \sum_{l=1}^{r_{4,j}}r_{6,j,l}'\right) = k'\,.
\end{equation}
Here we can further write
\begin{equation}
\partial_t^{r_{3,j}'}(f(u_1)-t) = \sum \partial_u^{r_{7,j}'}f(u_1) \prod_{l=1}^{r_{7,j}'} \partial_t^{r_{8,j,l}'}u_1 - \chi_{(r_{3,j}'=1)},\quad \sum_{l=1}^{r_{7,j}'}r_{8,j,l}' = r_{3,j}'\,,
\end{equation}
where $\chi_{(r_{3,j}'=1)}$ means 1 when $r_{3,j}'=1$ and 0 otherwise.

Therefore, the power of $t$ of the term (\ref{eqn:S_term1}) is
\begin{equation}
\begin{split}
& \frac{1}{2}-r_1' - (1+r_2+r_2') + \sum_{j=1}^{r_2'}\left(1-\frac{r_{7,j}'}{2}+\sum_{l=1}^{r_{7,j}'} (\frac{1}{2}-r_{8,j,l}')\right) \\
&+ \sum_{j=1}^{r_2}\left(1-\frac{r_{4,j}+r_{4,j}'}{2}+\sum_{l=1}^{r_{4,j}'}(\frac{1}{2}-r_{5,j,l}')+\sum_{l=1}^{r_{4,j}}(\frac{1}{2}-r_{6,j,l}')\right) \\
= & \frac{1}{2}-r_1' - (1+r_2+r_2') + \sum_{j=1}^{r_2'}(1-\sum_{l=1}^{r_{7,j}'} r_{8,j,l}') + \sum_{j=1}^{r_2}(1-\sum_{l=1}^{r_{4,j}'}r_{5,j,l}'-\sum_{l=1}^{r_{4,j}}r_{6,j,l}') \\
= & -\frac{1}{2}-r_1'  - \sum_{j=1}^{r_2'}r_{3,j}' - \sum_{j=1}^{r_2}(\sum_{l=1}^{r_{4,j}'}r_{5,j,l}'+\sum_{l=1}^{r_{4,j}}r_{6,j,l}') = -\frac{1}{2}-k'= \frac{1}{2}-(k'+1)\,.\\
\end{split}
\end{equation}
Notice that in the case $r_{3,j}'=1$ the term $ \sum \partial_u^{r_{7,j}'}f(u_1) \prod_{l=1}^{r_{7,j}'} \partial_t^{r_{8,j,l}'}u_1$ has $t$ power 0, the same as the term $\chi_{(r_{3,j}'=1)}=1$, thus the latter can be ignored. This finishes the induction for $(k,k'+1)$.
\end{proof}

\subsection{Proof of Proposition \ref{thm_interp}}
We first state a classical result from approximation theory (~\cite{Trefethen_ATAP}, Theorem 7.2)
\begin{lemma}\label{lem_tref}
 For an integer $m \ge 1$, let $f=f(z)$ and its derivatives through 
 $f^{(m -1)}$ be absolutely continuous on $[-1,1]$ and suppose the
 $m$-th derivative $\partial_z^m f$ is of bounded variation $V$. Then for any $n>m$, its Chebyshev interpolants (\ref{poly_interp}) satisfy
 $$ \|f-f^N\|_{L^\infty} \le {4V\over \pi m  (N-m )^m }. \eqno (7.5) $$
\end{lemma}

With this, we can show:
\begin{proof}[Proof of Proposition \ref{thm_interp}]
(\ref{thm_interp_1}) follows from Lemma \ref{lem_tref} by noticing that $V\le 2\|\partial_z^{m+1}f\|_{L^\infty}$ and $N-m\ge N/2$ if $N\ge 2m$. To see (\ref{thm_interp_2}), we use:
\begin{align*}
|\mathbb{E}(f)-\mathbb{E}(f^N)| &= |\int (f^N(z)-f(z))\pi(z)\rd{z}| \\
&\le \frac{C(m)\|\partial_z^{m+1}f\|_{L^\infty}}{N^m} \int\pi(z)\rd{z} = \frac{C(m)\|\partial_z^{m+1}f\|_{L^\infty}}{N^m}\,.
\end{align*}
To show~\eqref{thm_interp_3} is similar: 
\begin{equation}\label{thm_interp_3_1}
\begin{split}
|\textnormal{var}(f)-\textnormal{var}(f^N)| \le & \int |f^N(z)^2 - f(z)^2|\pi(z)\rd{z} + | (\mathbb{E}(f^N))^2-(\mathbb{E}(f))^2| \\ 
\le & \|f^N+f\|_{L^\infty}\|f^N-f\|_{L^\infty} + |\mathbb{E}(f^N) + \mathbb{E}(f)|\cdot|\mathbb{E}(f^N) - \mathbb{E}(f)| \\
\le & (2\|f\|_{L^\infty}+\|f^N-f\|_{L^\infty})\|f^N-f\|_{L^\infty} +(2\|f\|_{L^\infty}+\|f^N-f\|_{L^\infty})|\mathbb{E}(f^N) - \mathbb{E}(f)|\\ 
\le & \frac{C(m)(\|f^N-f\|_{L^\infty}+\|f\|_{L^\infty})\|\partial_z^{m+1}f\|_{L^\infty}}{N^m}\,,
\end{split}
\end{equation}
where in the third inequality we used 
\begin{equation}
\|f^N+f\|_{L^\infty} \le \|f^N-f\|_{L^\infty} + 2\|f\|_{L^\infty}\,,
\end{equation}
and in the last inequality we used (\ref{thm_interp_1}). To finally obtain~\eqref{thm_interp_3}, we define the piecewise linear function $f_1(z)$ by
\begin{equation}
f_1(z) = f(z_j) + \frac{f(z_{j+1})-f(z_j)}{z_{j+1}-z_j}(z-z_j),\quad \text{ for }z_j\le z < z_{j+1}\,,
\end{equation}
so that $f_1$ is absolutely continuous and satisfies $f_1(z_j)=f(z_j)$ for every $j$. Thus its Chebyshev interpolant is also $f^N$. Since
\begin{equation}
f_1'(z) = \frac{f(z_{j+1})-f(z_j)}{z_{j+1}-z_j},\quad \text{ for }z_j\le z < z_{j+1}\,,
\end{equation}
is a piecewise constant function, whose total variation is
\begin{equation}
\begin{split}
V = & \sum_{j=1}^{N-1}|\frac{f(z_{j+1})-f(z_j)}{z_{j+1}-z_j} - \frac{f(z_j)-f(z_{j-1})}{z_j-z_{j-1}}| \le  2\sum_{j=0}^{N-1} |\frac{f(z_{j+1})-f(z_j)}{z_{j+1}-z_j}| \\ & \le 4N\|f\|_{L^\infty} \sum_{j=0}^{N-1} \frac{1}{z_{j+1}-z_j} \le C\|f\|_{L^\infty}N^3 \,,
\end{split}
 \end{equation} 
where we used $z_{j+1}-z_j \ge \frac{C}{N^2}$, which is easily checked by using the mean value theorem for the function $\cos z$. Therefore, Lemma \ref{lem_tref} for $f_1$ with $m=1$ gives
\begin{equation}
\|f^N\|_{L^\infty} \le C\|f\|_{L^\infty}N^2\,,
\end{equation}
and thus
\begin{equation}
\|f^N-f\|_{L^\infty} \le C\|f\|_{L^\infty}N^2\,.
\end{equation}
This combined with (\ref{thm_interp_3_1}) gives (\ref{thm_interp_3}).
\end{proof}

\bibliographystyle{siam}
\bibliography{shock_random}

\end{document}